\documentclass[oneside,11pt]{amsart}
\usepackage[english,activeacute]{babel}

\usepackage[utf8]{inputenc}
\usepackage[T1]{fontenc}

\usepackage[normalem]{ulem}

 \usepackage{fullpage}

\usepackage{hyperref}
\usepackage{nicefrac}
\usepackage[foot]{amsaddr}

\usepackage{graphicx}
\usepackage{amssymb}
\usepackage{color}

\usepackage[runin]{abstract}

\usepackage{epstopdf}

\newtheorem{thm}{Theorem}[section]

\newtheorem{lem}[thm]{Lemma}

\newtheorem{prop}[thm]{Proposition}
\theoremstyle{definition}

\newtheorem{defn}[thm]{Definition}
\newtheorem{rmk}[thm]{Remark}

\newcommand{\ba}[1]{\begin{array}{#1}}
\newcommand{\ea}{\end{array}}

\newcommand{\R}{\mathbb{R}}
\newcommand{\Z}{\mathbb{Z}}
\newcommand{\N}{\mathbb{N}}
\newcommand{\C}{\mathbb{C}}

\newcommand{\al}{\alpha}
\newcommand{\eps}{\varepsilon}

\def\R{{\mathbf{R}}}

\def\Z{{\mathbf{Z}}}
\def\N{{\mathbf{N}}}

\newcommand{\grad}{\nabla}

\newcommand\Tu{\mathbb{T}}

\newcommand\grando[1]{\mathcal{O}(#1)}

\newcommand{\LBN}{
\text{  
{\Large
$\Delta\!\!\!\!\circ$
}}}

\newcommand{\sing}{\mathcal Z}
\newcommand{\inm}{M}
\newcommand{\mcil}{M_{\text{cylinder}}}
\newcommand{\mcono}{M_{\text{cone}}}
\newcommand{\distcil}{d_{\text{cylinder}}}
\newcommand{\distcono}{d_{\text{cone}}}
\newcommand{\clm}{M_\alpha}
\newcommand{\metr}{g}

\newcommand{\cc}[1]{{C^\infty_c\left(#1\right)}}
\newcommand{\dom}[1]{D(#1)}

\newcommand{\dmin}{D_\mathrm{min}}
\newcommand{\dmax}{D_{\mathrm{max}}}
\DeclareMathOperator\spn{span}
\newcommand{\hilb}{\mathcal H}

\newcommand{\phid}{\phi_D}
\newcommand{\phidp}{\phi_D^+}
\newcommand{\phidm}{\phi_D^-}
\newcommand{\phidpm}{\phi_D^\pm}
\newcommand{\phin}{\phi_N}
\newcommand{\phinp}{\phi_N^+}
\newcommand{\phinm}{\phi_N^-}
\newcommand{\phinpm}{\phi_N^\pm}

\newcommand{\delfz}{\widehat\Delta_0}
\newcommand{\delfzf}{(\widehat\Delta_0)_F}
\newcommand{\delfzcomp}{\widehat\Delta_0|_{\cc {\R\setminus\{0\}}}}
\newcommand{\delfk}{\widehat\Delta_k}
\newcommand{\dminf}{\dmin ({\delfz})}
\newcommand{\dmaxf}{\dmax ({\delfz})}
\newcommand{\lduef}{L^2(\R\setminus\{0\},|x|^{-\alpha}dx)}
\newcommand{\sobf}{H^1(\R\setminus\{0\},|x|^{-\alpha}dx)}
\newcommand{\sobfzero}{H^1_0(\R\setminus\{0\},|x|^{-\alpha}dx)}
\newcommand{\sobduef}{H^2(\R\setminus\{0\},|x|^{-\alpha}dx)}

\newcommand{\ldue}{L^2(M,\misura)}
\newcommand{\lduec}{L^2\left(M_\alpha,\misura\right)}
\newcommand{\sobzero}{H^1_0(M,\misura)}
\newcommand{\sob}{H^1(M,\misura)}
\newcommand{\sobc}{H^1\left(M_\alpha,\misura\right)}
\newcommand{\sobduec}{H^2\left(M_\alpha,\misura\right)}
\newcommand{\sobdue}{H^2(M,\misura)}
\newcommand{\sobduezero}{H^2_0(M,\misura)}

\DeclareMathOperator{\supp}{supp}

\newcommand{\delcomp}{\Delta|_{\cc \inm}}
\newcommand{\delf}{\Delta_F}
\newcommand{\delc}{\deltabridge}
\newcommand{\delb}{\deltabridge}
\newcommand{\deltabridge}{\Delta_B}
\newcommand{\deld}{\delf}
\newcommand{\delm}{\Delta_N}

\newcommand{\dex}{\mathcal E}
\newcommand{\dexf}{\mathcal E_F}
\newcommand{\dexb}{\mathcal E_B}
\newcommand{\dexm}{\mathcal E^+}
\newcommand{\fdexf}{\widehat{\mathcal E}_0}
\newcommand{\fdexfk}{\widehat{\mathcal E}_k}

\newcommand{\misura}{d\omega}

\newcommand{\bridge}{bridging }

\begin{document}

\title{Self-adjoint extensions and stochastic completeness of the Laplace-Beltrami operator on conic and anticonic surfaces}	
\author{Ugo~Boscain$^\dagger$ and Dario~Prandi$^\dagger$$^\ddagger$}
\address{$^\dagger$Centre National de Recherche Scientifique (CNRS), CMAP, \'Ecole Polytechnique, Route de Saclay, 91128 Palaiseau Cedex, France and Team GECO, INRIA-Centre de Recherche Saclay}
\address{$^\ddagger$LSIS, Université de Toulon, 83957 La Garde Cedex, France,}
\email{ugo.boscain@polytechnique.edu, dario.prandi@univ-tln.fr}
\thanks{
	This work was  supported by the European Research Council, ERC StG 2009 “GeCoMethods”, contract number 239748, and by the ANR project {\it GCM}, program ``Blanche'', project number NT09\_504490,  and by the Laboratoire d'Excellence Archimède, Aix-Marseille Université.
  }

\maketitle

\begin{abstract}
	We study the evolution of the heat and of a free quantum particle (described by the Schr\"odinger equation) on two-dimensional manifolds endowed with the degenerate Riemannian metric $ds^2=dx^2+|x|^{-2\alpha}d\theta^2$, where $x\in \R$, $\theta\in\Tu$ and the parameter $\alpha\in\R$.
	For $\alpha\le-1$ this metric describes cone-like manifolds (for $\alpha=-1$ it is a flat cone).
	For $\alpha=0$ it is a cylinder.
	For $\alpha\ge 1$ it is a Grushin-like metric.
	We show that the Laplace-Beltrami operator $\Delta$ is essentially self-adjoint
	if and only if $\alpha\notin(-3,1)$.
	In this case the only self-adjoint extension is the Friedrichs extension $\delf$, that does not allow communication through the singular set $\{x=0\}$ both for the heat and for a quantum particle. 
	For $\alpha\in(-3,-1]$ we show that for the Schr\"odinger equation only the average on $\theta$ of the wave function can cross the singular set, while the solutions of the only Markovian extension of the heat equation (which indeed is $\delf$) cannot.
	For $\alpha\in(-1,1)$ we prove that there exists a canonical self-adjoint extension $\deltabridge$, called \bridge extension, which is Markovian and allows the complete communication through the singularity (both of the heat and of a quantum particle).
	Also, we study the stochastic completeness (i.e., conservation of the $L^1$ norm for the heat equation) of the Markovian extensions $\delf$ and $\deltabridge$, proving that $\delf$ is stochastically complete at the singularity if and only if $\alpha\le -1$, while $\deltabridge$ is always stochastically complete at the singularity.

	\vspace{.4cm}

\end{abstract}

\section{Introduction}
In this paper we consider the Riemannian metric on $M=\big(\R\setminus\{0\}\big)\times\Tu$ whose orthonormal basis has the form:
\begin{eqnarray}
\label{eq-base}
X_1(x,\theta)=\left(\ba{c}1\\0 \ea\right),~~~X_2(x,\theta)=\left(\ba{c}0\\|x|^{\alpha} \ea\right),~~~\al\in\R.
\end{eqnarray}
Here $x\in\R\setminus\{0\}$, $\theta\in \Tu$ and $\alpha\in\R$ is a parameter. 
In other words we are interested in the Riemannian manifold $(M,\metr)$, where
\begin{eqnarray}
\label{eq-metric}
\metr=dx^2+|x|^{-2\al} d\theta^2, \mbox{ i.e., in matrix notation }\metr=\left(\ba{cc} 1&0\\0&|x|^{-2\al}\ea\right).
\end{eqnarray}
Define 
\begin{gather*}
	\mcil=\R\times\Tu, \qquad \mcono=\mcil/\sim,
\end{gather*}
where $(x_1,\theta_1)\sim(x_2,\theta_2)$ if and only if $x_1=x_2=0$.
In the following we are going to suitably extend the metric structure to $\mcil$ through \eqref{eq-base} when $\alpha\ge 0$, and to $\mcono$ through \eqref{eq-metric} when $\alpha< 0$.

Recall that, on a general two dimensional Riemannian manifold for which there exists a global orthonormal frame, the distance between two points can be defined equivalently as
\begin{multline}\label{eq:distcontr}
	d(q_1,q_2)=\inf\bigg\{ \int_0^1 \sqrt{u_1(t)^2+u_2(t)^2}\,dt \, \mid \gamma:[0,1]\to M \text{ Lipschitz },\gamma(0)=q_1,\,\gamma(1)=q_2 \\
	\text{ and } u_1,\,u_2 \text{ are defined by } \dot\gamma(t)=u_1(t)X_1(\gamma(t))+u_2(t)X_2(\gamma(t)) \bigg\},
\end{multline}
\begin{equation}\label{eq:distmetr}
	d(q_1,q_2)=\inf\left\{ \int_0^1 \sqrt{g_{\gamma(t)}(\dot\gamma(t),\dot\gamma(t))}\,dt \,\mid \gamma:[0,1]\to M \text{ Lipschitz },\gamma(0)=q_1,\,\gamma(1)=q_2 \right\},
\end{equation}
where $\{X_1,X_2\}$ is the global orthonormal frame for $(M,g)$.

{\bf Case $\alpha\ge 0$.}
Similarly to what is usually done in sub-Riemannian geometry (see e.g., \cite{noteandrei}), when $\alpha\ge 0$, formula \eqref{eq:distcontr} can be used to define a distance on $\mcil$ where $X_1$ and $X_2$ are given by formula \eqref{eq-base}.
We have the following, whose proof is contained in Appendix~\ref{sec:geom}).

\begin{lem}
	\label{lem:mcil}
	For any $\alpha\ge 0$, formula \eqref{eq:distcontr} endows $\mcil$ with a metric space structure, which is compatible with its original topology.
\end{lem}

{\bf Case $\alpha<0$.} In this case $X_1$ and $X_2$ are not well defined in $x=0$. 
However, one can extend the metric structure via formula \eqref{eq:distmetr} considering the metric $g$ given by \eqref{eq-metric}.
Since this metric identifies points on $\{x=0\}$, in the sense that they are at zero distance, formula \eqref{eq:distmetr} defines a well-defined metric space structure not to $\mcil$ but to $\mcono$.
Indeed, we have the following, proved in Appendix~\ref{sec:geom}.

\begin{lem}
	\label{lem:mcono}
	For $\alpha<0$, formula \eqref{eq:distmetr} endows $\mcono$ with a metric space structure, which is compatible with its original topology.
\end{lem}

\begin{rmk}[Notation]
	In the following we denote by $M_\alpha$ the generalized Riemannian manifold given as follows,
	\begin{itemize}
		\item $\alpha\ge0$: $M_\alpha=\mcil$ and metric structure induced by \eqref{eq-base};
		\item $\alpha< 0$: $M_\alpha=\mcono$ and metric structure induced by \eqref{eq-metric}.
	\end{itemize}
	The corresponding metric space is $(M_\alpha,d)$ and we let $\sing$ be the singular set, i.e., 
	\[
		\sing=
		\begin{cases}
		    \{0\}\times\Tu, &\qquad \alpha\ge 0,\\
		    \{0\}\times\Tu/\sim  &\qquad \alpha< 0.\\
		\end{cases}
	\]
	Observe that $\sing$ splits $M_\alpha$ in two sides $M^+=(0,+\infty)\times\Tu$ and $M^-=(-\infty,0)\times\Tu$.
\end{rmk}

Notice that, for $\alpha=1,2,3,\ldots$, $M_\alpha$ is an almost Riemannian structure in the sense of \cite{gaussbonnet,ars,bellaiche,arsnormal,lipschitzequiv}, while for $\alpha=-1,-2,-3,\ldots$ it corresponds to a singular Riemannian manifold with a semi-definite metric.

One of the main features of these metrics is that, except for $\al=0$, the corresponding Riemannian volumes have a singularity at $\sing$,
\begin{eqnarray*}
\misura=\sqrt{\det g} \,dx\,d\theta=|x|^{-\alpha}dx\,d\theta.
\end{eqnarray*}
Due to this fact, the corresponding Laplace-Beltrami operators contain some diverging first order terms,
\begin{eqnarray}
\label{eq-LB}
\Delta=\frac{1}{\sqrt{\det g}}\sum_{j,k=1}^2 \partial_j \left(\sqrt{\det g} \, g^{jk}\partial_k\right) =\partial_x^2+|x|^{2\alpha}\partial_\theta^2u-\frac{\alpha}{x}\partial_x
\end{eqnarray}

We have the following geometric interpretation of the metric  structure of $M_\alpha$ (see Figure \ref{fig:int-geo}).
For $\alpha=0$  the metric is the one of a cylinder, while for $\alpha=-1$ it is the one of a flat cone in polar coordinates.
For $\alpha<-1$, $M_\alpha$ is isometric to a surface of revolution $\mathcal S=\{ (t, r(t)\cos\vartheta, r(t)\sin\vartheta ) \mid t>0,\, \vartheta\in\Tu\}\subset\R^3$ with profile $r(t)\sim|t|^{-\alpha}$ as $|t|$ goes to zero.
For $\alpha>-1$ ($\alpha\neq 0$) it can be thought as a surface of revolution having a profile of the type $r(t)\sim |t|^{-\alpha}$ as $t\rightarrow 0$.
This interpretation for $\alpha>-1$ is only formal, since the embedding in $\R^3$ is deeply singular in a neighborhood of $t=0$.
Finally, the case $\alpha=1$ corresponds to the  Grushin metric on the cylinder. 
This geometric interpretation is explained in Appendix~\ref{sec:surfaces}.
\begin{figure}
	\includegraphics[width=\textwidth]{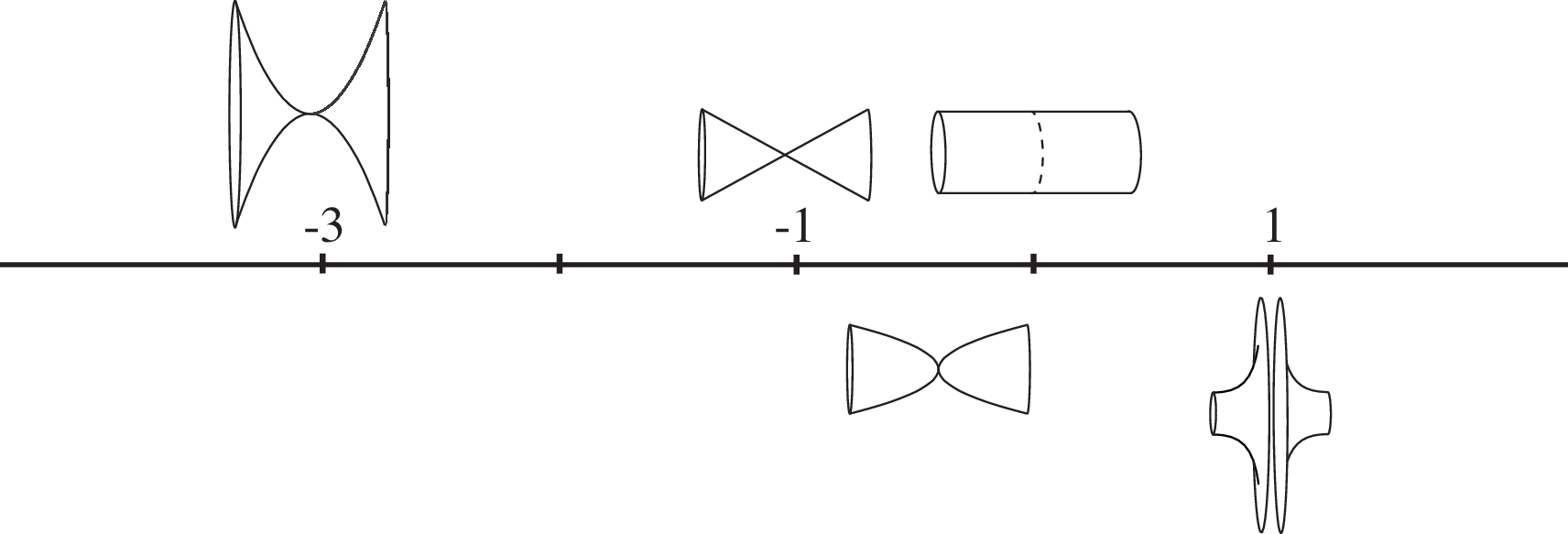}	
	\caption{Geometric interpretation of $M_\alpha$. The figures above the line are isometric to $M_\alpha$, while for the ones below the isometry is singular in a neighborhood of  $\sing$.}
	\label{fig:int-geo}
\end{figure}

\begin{rmk}
The curvature of $M_\alpha$ is given by
$
	K_\alpha(x)=- {\alpha(1+\alpha)}{x^{-2}}.
$
Notice that $M_\al$ and $M_\beta$ with $\beta=-(\alpha+1)$ have the same curvature for any $\alpha\in\R$
.
For instance, the cylinder with Grushin metric has the same curvature as the cone corresponding to $\alpha=-2$, but they are not isometric even locally (see \cite{arsnormal}).
\end{rmk}

\subsection{The problem} 
\label{sec:the_problem}

In this paper we are interested in the global behavior on $M_\alpha$ of the heat and free quantum particles, as modeled, respectively, by the heat and the Schr\"odinger equation
\begin{gather}
\label{eq-heat}
\partial_t \psi=\Delta\psi,\\
\label{eq-schr}
i \partial_t \psi=-\Delta \psi,
\end{gather}
where $\Delta$ is given by \eqref{eq-LB}.  
To give a meaning to these equations one needs to specify what $\Delta$ means on $\sing$, and to define in which function spaces we are working.  
In particular, it is natural to require $\Delta$ to be a self-adjoint operator on $L^2(M,\misura)$ (see Theorem~\ref{thm:stone}).
Thus, we will consider $\Delta|_\cc M$ and characterize all its self-adjoint extensions by prescribing opportune boundary conditions at the singularity $\sing$.

We will consider the following problems.
\begin{enumerate}
	 \renewcommand{\theenumi}{(Q\arabic{enumi})}
	 \renewcommand{\labelenumi}{\theenumi}
	\item \label{q1} Do the heat and free quantum particles flow through the singularity? 
	That is, there exists a self-adjoint extension of $\Delta|_\cc M$ such that, given an initial condition supported at time $t=0$ in $M^-$, is it possible that at time  $t>0$ the corresponding solution has some support in $M^+$?  \footnote{Notice that this is a necessary condition to have some positive controllability results by means of controls defined only on one side of the singularity, in the spirit of \cite{cannarsagrushin}. For a discussion of how hidden magnetic fields affect the self-adjointness of $\Delta$, we refer to \cite{aharonov}. For some results on 3D manifolds see \cite{moussa}.}

\item \label{q3} Given a self-adjoint extension of $\Delta$, does equation (\ref{eq-heat}) conserve the total heat (i.e. the $L^1$ norm of $\psi$)? 
This is equivalent to the problem of the stochastic completeness of $M_\alpha$, i.e.\ that the stochastic process defined by the diffusion $\Delta$ almost surely has infinite lifespan. 
In particular, we are interested in understanding if the heat is absorbed by the singularity.

For the Schr\"odinger equation only the question of conservation of total probability (i.e., the $L^2$ norm) has a physical meaning.
This question has trivial positive answer thanks to Stone's theorem.
\end{enumerate}

\begin{rmk}
By making the unitary change of coordinates in the Hilbert space $U:L^2(M,\misura)\to L^2(M,dxd\theta)$, defined by $Uv(x)=|x|^{-\nicefrac{\alpha}{2}}v(x)$, the Laplace-Beltrami operator is transformed in
\[
	\LBN=U\Delta U^{-1} = \partial_x^2 - \frac \alpha 2\left(1+\frac \alpha 2 \right)\frac{1}{x^2}+|x|^{2\alpha}\partial_\theta^2.
\]
This transformation was used to study the essential self-adjointness of $\Delta|_\cc{M}$ for $\alpha=1$ in \cite{boscainlaurent}.
Let us remark that, when acting on functions independent of $\theta$, the operator $\LBN$ reduces to the Laplace operator on $\R\setminus\{0\}$ in presence of an inverse square potential, usually called Calogero potential (see, e.g., \cite{gitman}).
\end{rmk}

\subsection{Self-adjoint extensions}

The problem of determining the self-adjoint extensions of $\delcomp$ on $\ldue$ has been widely studied in different fields. A lot of work has been done in the case $\alpha=-1$, in the setting of Riemannian manifolds with conical singularities (see e.g., \cite{cheeger,mooers}), and the same methods have been applied in the more general context of metric cusps or horns that covers the case $\alpha<-1$ (see e.g., \cite{cheeger2,bruning,lesch}). 
Concerning $\alpha>-1$, on the other hand, the literature regarding $\Delta$ is more thin (see e.g., \cite{morgan}). 

In the following we will consider only the real self-adjoint extensions, i.e., all the function spaces taken into consideration are composed of real-valued functions.
We refer to Appendix~\ref{app:schroedinger} for a discussion of the complex case.

Any closed symmetric extension $A$ of $\delcomp$ is such that $\dmin (\delcomp)\subset \dom A \subset \dmax (\delcomp)$, where the minimal and maximal domains are defined as
\begin{gather*}
	\dmin(\delcomp)=\dom {\overline{\Delta}} = \text{closure of } \cc M \text{ with respect to the norm } \|\Delta \cdot\|_{\lduec}+ \|\cdot\|_{\lduec},\\
	\dmax (\delcomp) = \dom {\Delta^*} = \{ u\in \lduec\colon\: \Delta u\in \lduec \text{ in the sense of distributions}\}.
\end{gather*}
Since it holds that $Au=\Delta^*u$ for any $u\in \dom A$, determining the self-adjoint extensions of $\delcomp$ amounts to classify the so-called domains of self-adjointness.
Following \cite{grigoryan}, we let the Sobolev spaces on the Riemannian manifold $M$ endowed with measure $\misura$, whose Riemannian gradient is $\nabla u(x,\theta)=(\partial_x u(x,\theta) , |x|^{2\alpha}\partial_\theta u(x,\theta))$, to be
\begin{gather*}
	\sob=\{ u \in \ldue\colon\: |\nabla u| \in \ldue \}, \quad \sobzero=\text{closure of }\cc M \text{ in } \sob,\\
	\sobdue=\{ u \in \sob\colon\:  \Delta u \in \ldue \}, \quad \sobduezero= \sobdue\cap\sobzero.
\end{gather*}
We define the Sobolev spaces $\sobc$ and $\sobduec$ in the same way.
We remark that with these conventions $\sobduezero$ is in general bigger than the closure of $\cc M$ in $\sobdue$.
Moreover, it may happen that $\sob=\sobzero$. 
Indeed this property will play an important role in the next section.
Proposition~\ref{prop:dmax}, contains a description of $\dmax(\delcomp)$ in terms of these Sobolev spaces.

Although in general the structure of the self-adjoint extensions of $\delcomp$ can be very complicated, the Friedrichs (or Dirichlet) extension $\deld$ is always well defined and self-adjoint.
Namely, 
\[
	\dom{\deld} = \sobduezero.
\]
Observe that, since $\ldue= L^2(M^+,\misura)\oplus L^2(M^-,\misura)$ and $\sobzero=H^1_0(M^+,\misura)\oplus H^1_0(M^-,\misura)$, it follows that
\[
	\dom{\deld}=\{ u\in H^1_0(M^+,\misura) \mid \Delta u \in L^2(M^+,\misura) \}\oplus \{ u\in H^1_0(M^-,\misura) \mid \Delta u \in L^2(M^-,\misura) \}.
\]
This implies that $\deld$ actually defines two separate dynamics on $M^+$ and on $M^-$ and, hence, that there is no hope for an initial datum concentrated in $M^+$ to pass to $M^-$, and vice versa. 
This proves that, if $\Delta|_ {\cc M}$ is essentially self-adjoint (i.e., the only self-adjoint extension is $\deld$) the question \ref{q1} has a negative answer. 

\subsubsection{Essential self-adjointness of $\delcomp$}
The rotational symmetry of the cones suggests to proceed by a Fourier decomposition in the $\theta$ variable, considering the corresponding orthonormal basis $\{ e_k \}_{k\in\Z}\subset L^2(\Tu)$ and the decomposition $\ldue=\bigoplus_{k=-\infty}^{\infty}H_k$, $H_k \cong \lduef$.
Then, 
\begin{equation}
	\label{eq:delfk}
	\Delta = \bigoplus_{k=-\infty}^{+\infty} \widehat \Delta_k,
	\qquad\text{where}\qquad
	\widehat \Delta_k = \partial_x^2 - \frac \alpha x \partial_x -|x|^{2\alpha}k^2.
\end{equation}
As proved in Proposition~\ref{prop:fourier}, the essential self-adjointness of all the operators $\widehat \Delta_k$ on $\cc {\R\setminus\{0\}}$ is a sufficient condition for  $\delcomp$ to satisfy the same property. 

The following theorem extends a result in \cite{boscainlaurent} by classifying the essential self-adjointness of $\delcomp$ and its Fourier components and is proved in Section~\ref{sec:proof}.
We remark that the same result holds if $\delcomp$ acts on complex-valued functions (see Theorem~\ref{thm:sacompl}).

\begin{thm}\label{thm:sa}
    Consider $M_\alpha$ for $\alpha \in\R$ and the corresponding Laplace-Beltrami operator $\delcomp$ as an unbounded operator on $\ldue$.
    Then,
    \begin{enumerate}
    \renewcommand{\theenumi}{\roman{enumi}}
        \item If $\alpha\le -3$ the operator $\delcomp$ is essentially self-adjoint;
        \item if $\alpha\in(-3,-1]$, only its first Fourier component $\widehat \Delta_0$ is not essentially self-adjoint;
        \item if $\alpha\in(-1,1)$, all the Fourier components of $\delcomp$ are not essentially self-adjoint;
        \item if $\alpha\ge1$ the operator $\delcomp$ is essentially self-adjoint.
    \end{enumerate}
\end{thm}

As a corollary of this theorem, we get the following answer to \ref{q1}.

\begin{center}
\begin{tabular}{c|c}
$\alpha \le -3$ & Nothing can flow through $\sing$ \\ \hline
$-3< \alpha \le -1$ & Only the average over $\Tu$ of the function can flow through $\sing$ \\ \hline
$-1< \alpha < 1$ & It is possible to have full communication between the two sides \\ \hline
$1\le \alpha $ & Nothing can flow through $\sing$ \\ \hline
\end{tabular}
\end{center}

\begin{rmk}
	When $\alpha\in(-3,0)$, since the singularity reduces to a single point, one would expect to be able to ``transmit'' through $\sing$ only the average over $\Tu$ of a function.
	Theorem~\ref{thm:sa} shows that this is the case for $\alpha\in(-3,-1]$, but not for $\alpha\in(-1,0)$.
	Looking at $M_\alpha$, $\alpha\in(-1,0)$, as a surface embedded in $\R^3$ the possibility of transmitting Fourier components other than $k=0$, is due to the deep singularity of the embedding.
	In this case we say that the contact between $M^+$ and $M^-$ is {\em non-apophantic}.
\end{rmk}

In the following we will describe the self-adjoint extension realising the maximal communication between the two sides, which we call the \emph{bridging extension}, in order to have a more precise answer to \ref{q1} for $\alpha\in (-3,1)$.
In particular, to identify this extension when $\alpha\in(-3,-1]$, it will be sufficent to study only the first Fourier component.
Indeed, by Theorem~\ref{thm:sa}, for these values of $\alpha$ it is possible to decompose any self-adjoint extension $A$ of $\delcomp$ as
\begin{equation}
	\label{eq:decompfourier}
	A= \widehat A_0 \oplus \left( \bigoplus_{k\in\Z\setminus\{0\}} \delfk \right).
\end{equation}
Here, $\widehat A_0$ is a self-adjoint extension of $\delfz$ and, with abuse of notation, we denoted the only self-adjoint extension of $\delfk$ by $\delfk$ as well.

\subsubsection{The first Fourier component $\delfz$}
In this section we describe the real self-adjoint extensions of $\delfzcomp$ on $\lduef$ when $\alpha\in(-3,1)$.
For a description of its complex self-adjoint extensions, we refer to Theorem~\ref{thm:saecompl}.
Observe that, since this operator is regular at the origin in the sense of Sturm-Liouville problems (see Definition~\ref{defn:regular}) if and only if $\alpha>-1$, for $\alpha\le-1$ the boundary conditions will be asymptotic, and not punctual.

Let $\phidp$ and $\phinp$ be two smooth functions on $\R\setminus\{0\}$, supported in the interval $(0,2)$, and such that, for any $x\in(0,1]$ it holds
\begin{equation}
	\label{def:phidn}
    \phidp(x) = 1 , \qquad \phinp(x) = 
	\begin{cases}
	    ({1+\alpha})^{-1}\,{x^{1+\alpha}} &\qquad \text{if } \alpha\neq-1,\\
	    \log(x) &\qquad \text{if } \alpha=-1.	    
	\end{cases}
\end{equation}
Let also $\phidm(x)=\phidp(-x)$ and $\phinp(x)=\phinm(-x)$. 
Finally, recall that, on $\R\setminus\{0\}$ endowed with the Euclidean structure, the Sobolev space ${\sobduef}$ is the space of functions $u\in\lduef$ such that $|\partial_x u|, |\partial^2_x u| \in\lduef$.
Then, the following holds.
 
\begin{thm} \label{thm:sae}
	Let $\dminf$ and $\dmaxf$ be the minimal and maximal domains of $\delfzcomp$ on $\lduef$, for $\alpha\in (-3,1)$.
	Then, 
	\begin{gather*}
		\dminf =  \text{closure of } \cc {\R\setminus\{0\}} \text{ in } H^2(\R\setminus\{0\},|x|^{-\alpha}dx)\\
		\dmaxf = \{ u=u_0+u^+_D\phidp + u^+_N \phinp + u^-_D\phidm + u^-_N \phinm \colon\: u_0\in\dminf \text{ and } u^\pm_D,\,u_N^\pm\in\R\},
	\end{gather*}
	Moreover, $A$ is a self-adjoint extension of $\delfz$ if and only if $Au=(\delfz)^*u$, for any $u\in\dom A$, and one of the following holds
	\begin{enumerate}
 	\renewcommand{\theenumi}{\roman{enumi}}
		\item \em Disjoint dynamics: \em there exist $c_+,c_-\in (-\infty,+\infty]$ such that 
			\[
				\dom A = \big\{ u\in\dmaxf\colon\: u_N^+=c_+u^+_D \text{ and } u_N^-=c_- u^+_D\big\}.
			\]
		\item \em Mixed dynamics: \em there exist $K\in SL_2(\R)$ such that
			\[
				\dom A = \big\{ u\in\dmaxf \colon\: (u_D^-,u_N^-)^T= K \, (u_D^+,u_N^+)^T \big\}.
			\]
	\end{enumerate}
	Finally, the Friedrichs extension $(\delfz)_F$ is the one corresponding to the disjoint dynamics with $c_+=c_-=0$ if $\alpha\le -1$ and with $c_+=c_-=+\infty$ if $\alpha>-1$.
\end{thm}

From the above theorem (see Remark~\ref{rmk:dirneu}) it follows that $u_N^\pm=\lim_{x\rightarrow 0^\pm} |x|^{-\alpha}\,\partial_x u(x)$ and, if $-1<\alpha<1$, that $u_D^\pm=u(0^\pm)$.
Moreover, the last statement implies that 
\[
	\dom{(\delfz)_F}=
		\begin{cases}
			\{u\in\dmaxf\colon\: u_N^+=u_N^-=0\} \qquad &\text{if }\alpha\le-1,\\
			\{u\in\dmaxf\colon\: u(0^+)=u(0^-)=0\}\qquad &\text{if }\alpha>-1.
		\end{cases}
\]
In particular, if $\alpha\le -1$ the Friedrichs extension does not impose zero boundary conditions.

Clearly, the evolutions associated with the disjoint dynamics extensions yield a negative answer to \ref{q1}. On the other hand, the mixed dynamics extensions permit information transfer between the two halves of the space. 
Since to classify the self-adjoint extensions for $\alpha\in(-3,-1]$ it is enough to study $\delfz$, this analysis completes the classification for this case.
On the other hand, since for $\alpha\in(-1,1)$ all the Fourier components are not essentially self-adjoint, a complete classification requires more sophisticated techniques. 
We will, in turn, study some selected extensions.

\begin{rmk}
The mixed dynamics extension with $K=\text{Id}$ is the \emph{bridging extension} of the first Fourier component, which we will denote by $(\delfz)_B$.
If $\alpha\in(-3,-1]$, the bridging extension $\delb$ of $\delcomp$ is then defined by \eqref{eq:decompfourier} with $A_0=(\delfz)_B$.
The bridging extension for $\alpha\in(-1,1)$ is described in the following section.
\end{rmk}

\subsection{Markovian extensions}

It is a well known result, that each non-positive self-adjoint operator $A$ on an Hilbert space $\hilb$ defines a strongly continuous contraction semigroup, denoted by $\{e^{tA}\}_{t\ge0}$, see Theorem~\ref{thm:stone}. If $\hilb=\ldue$ and it holds $0\le e^{tA} u \le 1$ $\misura$-a.e. whenever $u\in \ldue$, $0\le u \le 1$ $\misura$-a.e., the semigroup $\{e^{tA}\}_{t\ge0}$ and the operator $A$ are called Markovian. 
The interest for Markov operators lies in the fact that, under an additional assumption which is always satisfied in the cases we consider (see Section~\ref{sec:bilinear}), Markovian operators are generators of Markov processes $\{X_t\}_{t\ge 0}$ (roughly speaking, stochastic processes which are independent of the past).

Since essentially bounded functions are approximable from $\ldue$, the Markovian property allows to extend the definition of $e^{tA}$ from $\ldue$ to  $L^\infty(M,\misura)$. Let $1$ be the constant function $1(x,\theta)\equiv1$. Then \ref{q3} is equivalent to the following property.

\begin{defn}
	\label{def:stoc}
	A Markovian operator $A$ is called \em stochastically complete \em (or conservative) if $e^{tA}1=1$, for any $t>0$. 
	It is called \emph{explosive} if it is not stochastically complete.
\end{defn}

It is well known that this property is equivalent to the fact that the Markov process $\{X_t\}_{t\ge0}$, with generator $A$, if it exists, has almost surely infinite lifespan. 

For any $u\in L^2(M,d\omega)$ and $t>0$, define
\begin{equation*}
	S_t u = \int_0^t e^{tA}u\, ds.
\end{equation*}
The family $\{S_t\}_{t>0}$ is a well-defined family of symmetric operators on $L^2(M,d\omega)$ which can be extended to $L^1(M,dw)$ and satisfies $S_t u(q)\le S_{t'}u(q)$ for $d\omega$-a.e.\ $q\in M$ whenever $t'>t$.
Then, we let
\begin{equation*}
	Gu(q) = \lim_{n\rightarrow+\infty} S_n u(q) \in [0,+\infty],
\end{equation*}
and pose the following.

\begin{defn}\label{def:recurr}
	A Markovian operator is called \emph{recurrent} if $0<Gu(q)<+\infty$ for $d\omega$-a.e.\ $q\in M$ and $u\in L^1(M,d\omega)$ positive.
\end{defn}

When $A$ is the generator of a Markov process $\{X_t\}_{t\ge0}$, the recurrence of $A$ is equivalent to 
\[
	\mathbb P_q \{ \text{there exists a sequence } t_n\rightarrow +\infty \text{ such that } X_{t_n}\in \Omega \} =1,
\]
for any set $\Omega$ of positive measure and any point $q\in M$.
Here $\mathbb P_q$ denotes the measure in the space of paths emanating from a point $q $ associated to $\{X_t\}_{t\ge0}$. 

\begin{rmk}
	As it is suggested from the probabilistic interpretation, recurrence of an operator implies its stochastic completeness. 
	Equivalently, any explosive operator is not recurrent.
\end{rmk}

We are particularly interested in distinguish how the stochastically completeness and the recurrence are influenced by the singularity $\sing$ or by the behavior at $\infty$. 
Thus we will consider the manifolds with borders $M_0=M\cap ([-1,1]\times \Tu)$ and $M_\infty=M\setminus [-1,1]\times \Tu$, with Neumann boundary conditions.
Indeed, with these boundary conditions, when the Markov process $\{X_t\}_{t\ge 0}$ hits the boundary it is reflected, and hence the eventual lack of recurrence or stochastic completeness on $M_0$ (resp. on $M_\infty$) is due to the singularity $\sing$ (resp. to the behavior at $\infty$).
According to Definition~\ref{def:stocorigin}, if a Markovian operator $A$ on $M$ is recurrent (resp. stochastically complete) when restricted on $M_0$ we will call it recurrent (resp. stochastically complete) at $0$.
Similarly, when the same happens on $M_\infty$, we will call it recurrent (resp. stochastically complete) at $\infty$.
As proven in Proposition~\ref{prop:unione}, a Markovian extension of $\delcomp$ is recurrent (resp. stochastically complete) if and only if it is recurrent (resp. stochastically complete) both at $0$ and at $\infty$.

In this context, it makes sense to give special consideration to three specific self-adjoint extensions, corresponding to different conditions at $\sing$.
Namely, we will consider the already mentioned Friedrichs extension $\deld$, that corresponds to an absorbing condition, the Neumann extension $\delm$, that corresponds to a reflecting condition, and the \bridge extension $\delb$, that corresponds to a free flow through $\sing$ and is defined only for $\alpha\in(-1,1)$.
In particular, the latter two have the following domains (see Proposition~\ref{prop:domains}),
\begin{align*}
	&\dom {\delm} = \{ u\in\sob \mid (\Delta u,v)=(\nabla u,\nabla v) \text{ for any } v\in \sob\},\\
	&\dom {\delc} = \{ \sobduec \mid u(0^+,\cdot)=u(0^-,\cdot),\, \lim_{x\rightarrow 0^+} |x|^{-\alpha}u(x,\cdot)=\lim_{x\rightarrow 0^-} |x|^{-\alpha}u(x,\cdot) \text{ for a.e. }\theta\in\Tu\}.
\end{align*}
Each one of $\delf$, $\delm$ and $\delb$ is a self-adjoint Markovian extension.
However, it may happen that $\delf=\delm$. In this case $\delf$ is the only Markovian extension, and the operator $\delcomp$ is called Markov unique.
This is the case, for example, when $\delcomp$ is essentially self-adjoint.

The following result, proved in Section~\ref{sec:stoccompl}, will answer to \ref{q3}.

\begin{thm} \label{thm:markov}
  Consider $M$ endowed with the Riemannian metric defined in (\ref{eq-metric}), for $\alpha\in\R$, and $\delcomp$ be the corresponding Laplace-Beltrami operator, defined as an unbounded operator on $\ldue$.
    Then the following holds.
    \begin{enumerate}
    \renewcommand{\theenumi}{\roman{enumi}}
        \item If $\alpha< -1$ then $\delcomp$ is Markov unique, and $\deld$ is stochastically complete at $0$ and recurrent at $\infty$;
        \item if $\alpha=-1$ then $\delcomp$ is Markov unique, and $\deld$ is recurrent both at $0$ and at $\infty$;
        \item if $\alpha\in(-1,1)$, then $\delcomp$ is not Markov unique and, moreover,
        	\begin{enumerate}
        	    \item any Markovian extension of $\delcomp$ is recurrent at $\infty$, 
        	    \item $\delf$ is explosive at $0$, while both $\delb$ and $\delm$ are recurrent at $0$,
        	\end{enumerate}
        \item if $\alpha\ge1$ then $\delcomp$ is Markov unique, and $\deld$ is explosive at $0$ and recurrent at $\infty$;
    \end{enumerate}
\end{thm}

In particular, Theorem~\ref{thm:markov} implies that for $\alpha\in(-3,-1]$ no mixing behavior defines a Markov process. 
On the other hand, for $\alpha\in(-1,1)$ we can have a plethora of such processes. 

\begin{rmk}
	Since the singularity $\sing$ is at finite distance from any point of $M$,
	one can interpret a Markov process that is explosive at $0$ as if $\sing$ were absorbing the heat.
\end{rmk}

As a corollary of \ref{thm:markov}, we get the following answer to \ref{q3}.
\begin{center}
\begin{tabular}{c|c}
$\alpha \le -1$ & The heat is not absorbed by $\sing$ \\ \hline
$-1< \alpha < 1$ & The Friedrichs extension is absorbed by $\sing$, \\
 & while the Neumann and the \bridge extensions are not. \\ \hline
$1\le \alpha $ & The heat is absorbed by $\sing$ \\ \hline
\end{tabular}
\end{center}

\subsection{Structure of the paper} 
\label{sub:structure_of_the_paper}

The structure of the paper is the following. 
In Section~\ref{sec:sae}, after some preliminaries regarding self-adjointness, we analyze in detail the Fourier components of the Laplace-Beltrami operator on $M_\alpha$, proving Theorems~\ref{thm:sa} and \ref{thm:sae}.
We conclude this section with a description of the maximal domain of the Laplace-Beltrami operator in terms of the Sobolev spaces on $M_\alpha$, contained in Proposition~\ref{prop:dmax}.

Then, in Section~\ref{sec:bilinear}, we introduce and discuss the concepts of Markovianity, stochastic completeness and recurrence through the potential theory of Dirichlet forms.
After this, we study the Markov uniqueness of $\delcomp$ and characterize the domains of the Friedrichs, Neumann and \bridge extensions (Propositions~\ref{prop:onlyfried} and \ref{prop:domains}).
Then, we define stochastic completeness and recurrence at $0$ and at $\infty$, and, in Proposition~\ref{prop:zero}, we discuss how these concepts behave if the $k=0$ Fourier component of the self-adjoint extension is itself self-adjoint.
In particular, we show that the Markovianity of such an operator $A$ implies the Markovianity of its first Fourier component $\widehat A_0$, and that the stochastic completeness (resp. recurrence) at $0$ (resp. at $\infty$) of $A$ and $\widehat A_0$ are equivalent.
Then, in Proposition~\ref{prop:unione} we prove that stochastic completeness or recurrence are equivalent to stochastically completeness or recurrence both at $0$ and at $\infty$.
Finally, we prove Theorem~\ref{thm:markov}.

The proofs of Lemmata~\ref{lem:mcil} and \ref{lem:mcono} are contained in Appendix~\ref{sec:geom}, while in Appendix~\ref{sec:surfaces} we justify the geometric interpretation of Figure~\ref{fig:int-geo}.
Appendix~\ref{app:schroedinger} contains the description of the complex self-adjoint extension of $\delfz$.

\section{Self-adjoint extensions}
\label{sec:sae}

\subsection{Preliminaries} 

Let $\hilb$ be an Hilbert space with scalar product $(\cdot,\cdot)_\hilb$ and norm $\|\cdot\|_\hilb=\sqrt{(\cdot,\cdot)_\hilb}$. 
Given an operator $A$ on $\hilb$ we will denote its domain by $\dom A$ and its adjoint by $A^*$.
Namely, if $A$ is densely defined, $\dom {A^*}$ is the set of $\varphi \in \hilb$ such that there exists $\eta\in\hilb$ with
	$(A\psi,\varphi)_\hilb=(\psi,\eta)_\hilb, \text{ for all } \psi\in\dom A.$
For each such $\varphi$, we define $A^*\varphi = \eta$.

Given two operators $A,B$, we say that $B$ is an extension of $A$ (and we will write $A\subset B$) if $\dom A\subset \dom B$ and $A\psi=B\psi$ for any $\psi\in\dom A$.
A densely defined operator $A$ is \emph{symmetric} if $A\subset A^*$, i.e., if
\[
	(A\psi,\varphi)_\hilb = (\psi,A\varphi)_\hilb, \qquad \text{for all } \psi\in\dom A.
\]
A densely defined operator $A$  is \em self-adjoint \em if $A=A^*$, that is if it is symmetric and $\dom A=\dom {A^*}$, and is \em non-positive \em if $(A\psi,\psi)\le 0$ for any $\psi\in\dom A$. 

Given a strongly continuous group $\{T_t\}_{t\in\R}$ (resp. semigroup $\{T_t\}_{t\ge0}$), its generator $A$ is defined as
\begin{equation*}
	A u = \lim_{t\rightarrow0} \frac {T_t u - u} {t}, \quad 
	\dom A = \{ u \in\hilb\mid Au \text{ exists as a strong limit}\}.
\end{equation*}
When a group (resp. semigroup) has generator $A$, we will write it as $\{e^{tA}\}_{t\in\R}$ (resp. $\{e^{tA}\}_{t\ge0}$).
Then, by definition, $u(t)= e^{tA}u_0$ is the solution of the functional equation 
\[
	\left\{
	\begin{array}{l}
		\partial_t u(t) = Au(t) \\
		u(0)= u_0\in\hilb.
	\end{array}	
	\right.
\]
Recall the following classical result.
\begin{thm}\label{thm:stone}
    Let $\hilb$ be an Hilbert space, then
    \begin{enumerate}
    	\item \em(Stone's theorem)\em  The map $A\mapsto \{e^{ i tA}\}_{t\in\R}$ induces a one-to-one correspondence 
    		\[
    			A \text{ self-adjoint operator } \Longleftrightarrow \{e^{ i tA}\}_{t\in\R} \text{  strongly continuous unitary group};
    		\]
    	\item The map $A\mapsto \{e^{tA}\}_{t\ge 0}$ induces a one-to-one correspondence 
    		\begin{equation*}
   				\begin{split}
    			A \text{ non-positive self-adjoint operator } \Longleftrightarrow 
    			\{e^{tA}\}_{t\ge 0} &\text{ strongly continuous contraction}\\
    			& \text{semigroup of self-adjoint operators};	
   				\end{split}
    		\end{equation*}
    \end{enumerate}
\end{thm}

For any Riemannian manifold ${\mathcal M}$ with Riemannian volume $dV$, Green's identity implies that  $\Delta|_{\cc {\mathcal M}}$ is symmetric.
However, from the same formula, follows that 
\[
	\dom {\Delta|_{\cc {\mathcal M}} \,^*} = \{ u\in L^2({\mathcal M},dV)\mid \Delta u\in L^2({\mathcal M},dV) \} \varsupsetneqq \cc {\mathcal M},
\]
where $\Delta u$ is intended in the sense of distributions.
Hence, $\Delta$ is not self-adjoint on $\cc {\mathcal M}$.

Since, by Theorem~\ref{thm:stone}, in order to have a well defined solution of the Schr\"odinger equation the Laplace-Beltrami operator has to be self-adjoint, we have to extend its domain in order to satisfy this property. 
For the heat equation, on the other hand, we will need also to worry about the fact that it stays non-positive while doing so.
We will tackle this problem in the next section, where we will require the stronger property of being Markovian (i.e., that the evolution preserves both the non-negativity and the boundedness).

The simplest extension one can build for a symmetric operator $A$ is the closure $\bar A$. 
Namely, $\dom {\bar A}$ is the closure of $\dom A$ with respect to the graph norm $\|\cdot\|_A= \|A\cdot\|_\hilb+\|\cdot\|_\hilb$, and $\bar A \psi = \lim_{n\rightarrow +\infty} A\psi_n$ where $\{\psi_n\}_{n\in\N}\subset \dom A$ is such that $\psi_n\rightarrow \psi$ and $\{A\psi_n\}_n$ is a Cauchy sequence in $\hilb$.
Observe that $A\subset \bar A\subset A^*$, and hence any self-adjoint extension $B$ of $A$ will be such that ${\bar A}\subset  B\subset  {A^*}$. 
For this reason, we let $\dmin(A)=\dom {\bar A}$ and $\dmax (A) = \dom {A^*}$.
Moreover, from this fact follows that any self-adjoint extension $B$ will be defined as $B\psi=A^*\psi$ for $\psi\in \dom B$, so we are only concerned in specifying the domain of $B$.
The simplest case is the following.

 \begin{defn}
    A symmetric operator is called \em essentially self-adjoint \em if its closure is self-adjoint.
 \end{defn}

It is a well known fact, dating as far back as the series of papers \cite{gaffney1,gaffney2}, that the Laplace-Beltrami operator is essentially self-adjoint on any complete Riemannian manifold.
On the other hand, it is clear that if the manifold is incomplete this is no more the case, in general (see \cite{masamune, grigoryanmasamune}). 
It suffices, for example, to consider the case of an open set $\Omega\subset\R^n$, where to have the self-adjointness of the Laplacian, we have to pose boundary conditions (Dirichlet, Neumann or a mixture of the two).
In our case, Theorem~\ref{thm:sa} will give an answer to the problem of whether $\delcomp$ is essentially self-adjoint or not.

\subsection{Fourier decomposition and self-adjoint extensions of Sturm-Liouville operators}

There exist various theories allowing to classify the self-adjoint extensions of symmetric operators. 
We will use some tools from the Neumann theory (see \cite{reedsimon}) and, when dealing with one-dimensional problems, from the Sturm-Liouville theory.
Let $\hilb$ be a complex Hilbert space and $i$ be the imaginary unit.
The \emph{deficiency indexes} of $A$ are then defined as
\[
	n_+(A)=\dim \ker (A+ i), \qquad n_-(A)=\dim \ker (A- i).
\]
Then $A$ admits self-adjoint extensions if and only if $n_+(A)=n_-(A)$, and they are in one to one correspondence with the set of partial isometries between $\ker(A- i)$ and $\ker(A +  i)$.
Obviously, $A$ is essentially self-adjoint if and only if $n_+(A)=n_-(A)=0$.

Following \cite{zettl}, we say that a self-adjoint extension $B$ of $A$ in $\hilb$ is a \emph{real self-adjoint extension} if $u\in\dom B$ implies that $\overline u \in \dom B$ and $B(\overline u)=\overline{Bu}$.
When $\hilb=\ldue$, i.e.  the real Hilbert space of square-integrable real-valued functions on $M$, the self-adjoint extensions of $A$ in $\ldue$ are the restrictions to this space of the real self-adjoint extensions of $A$ in $L^2_\C(M,\misura)$, i.e. the complex Hilbert space of square-integrable complex-valued functions.
This proves that $A$ is essentially self-adjoint in $\ldue$ if and only if it is essentially self-adjoint in $L^2_\C(M,\misura)$.
Hence, when speaking of the deficiency indexes of an operator acting on $\ldue$, we will implicitly compute them on $L^2_\C(M,\misura)$.

We start by proving the following general proposition that will allow us to study only the Fourier components of $\delcomp$, in order to understand its essential self-adjointness.

 \begin{prop}\label{prop:fourier}
 	Let $A_k$ be symmetric on $\dom{A_k}\subset H_k$, for any $k\in\Z$ and let $\dom A$ be the set of vectors in $\hilb=\bigoplus_{k\in \Z} H_k$ of the form $\psi=(\psi_1,\psi_2,\ldots)$, where $\psi_k\in \dom {A_k}$ and all but finitely many of them are zero.
 	Then $A=\sum_{k\in\Z}A_k$ is symmetric on $\dom A$, $n_+(A)=\sum_{k\in \Z} n_+(A_k)$ and $n_-(A)=\sum_{k\in \Z} n_-(A_k)$.
 \end{prop}

 \begin{proof}
 	Let $\psi=(\psi_1,\psi_2,\ldots)\in \dom A$. Then, by symmetry of the $A_k$'s and the fact that only finitely many $\psi_k$ are nonzero, it holds
 	\[
 		(Au,v)_\hilb = \sum_{k\in\Z} (A_ku_k,v_k)_{H_k}=\sum_{k\in\Z} (u_k,A_kv_k)_{H_k}=(u,Av)_\hilb.
 	\]
 	This proves the symmetry of $A$.
 	
  	Observe now that $\psi=(\psi_1,\psi_2,\ldots)\in \ker (A\pm i)$ if and only if $0=A\psi\pm i = (A_1\psi_1\pm i,A_2\psi_2\pm i,\ldots)$. 
  	This clearly implies that $\dim \ker(A\pm i)=\sum_{k\in\Z}\dim \ker(A_k\pm i)$, completing the proof.
 \end{proof}

Since the Fourier components $\delfk$ defined in \eqref{eq:delfk} are second order differential operators of one variable, they can be studied via Sturm-Liouville theory.
Let $J=(a_1,b_1)\cup(a_2,b_2)$, $-\infty\le a_1 < b_1\le a_2 < b_2\le +\infty$, and for $1/p,\,q,\,w\in L^1_{\text{loc}}(J)$
consider the Sturm-Liouville operator on $L^2(J,w(x)dx)$ defined by
\begin{equation}\label{eq:slo}
	Au =  \frac 1 w\bigg( -\partial_x(p \,\partial_x u)+  q u \bigg).
\end{equation}
Letting $J=\R\setminus\{0\}$, $w(x)=|x|^{-\alpha}$, $p(x)=-|x|^{-\alpha}$, and $q(x)=-k^2|x|^\alpha$, we recover $\widehat \Delta_k$.
In the following  we will heavily rely on \cite[Chapter~13]{zettl}, where self-adjointess of Sturm-Liouville operators defined on a disjoint union of two connected intervals is studied. 

For a Sturm-Liouville operator the maximal domain can be explicitly characterized as
\begin{equation}
	\label{eq:dmaxsturm}
	\dmax(A)=\{ u:J\to \R\mid u,\, p\,\partial_x u \text{ are absolutely continuous on } J,\text{ and } u,\, Au \in L^2(J,w(x)dx)\}.
\end{equation}
In (\ref{eq:patching}), at the end of the section, we will give a precise characterization of the minimal domain.

\begin{defn}
The endpoint (finite or infinite) $a_1$, is limit-circle if all solutions of the equation $Au=0$ are in $L^2((a_1,d),w(x)dx)$ for some (and hence any) $d\in (a_1,b_1)$.
Otherwise $a_1$ is limit-point.

Analogous definitions can be given for $b_1$, $a_2$ and $b_2$.
\end{defn}

Let us define the Lagrange parenthesis of $u,v:J\to\R$ associated to \eqref{eq:slo} as the bilinear antisymmetric form
\[
	[u,v]=u \, p\, \partial_x v - v\,p\,	 \partial_x u.
\]
By \cite[(10.4.41)]{zettl} or \cite[Lemma~3.2]{weyl}, we have that $[u,v](d)$ exists and is finite for any $u,v\in\dmax(\delfk)$ and any endpoint $d$ of $J$.

\begin{defn}\label{defn:regular}
    The Sturm-Liouville operator \eqref{eq:slo} is \em regular \em at the endpoint $a_1$ if for some (and hence any) $d\in(a_1,b_1)$, it holds
    \[
    	\frac 1 p,\, q,\, w \in L^1((a_1,d)).
    \] 
    A similar definition holds for $b_1,\,a_2,\, b_2$.
\end{defn}

In particular, for any $k\in\Z$, the operator $\delfk$ is never regular at the endpoints $+\infty$ and $-\infty$, and is regular at $0^+$ and $0^-$ if and only if $\alpha\in(-1,1)$.

We will need the following theorem, that we state only for real extensions and in the cases we will use.

\begin{thm}[Theorem 13.3.1 in \cite{zettl}]\label{thm:zettl}
	Let $A$ be the Sturm-Liouville operator on $L^2(J,w(x)dx)$ defined in \eqref{eq:slo}.
	Then
	\[
		n_+(A)=n_-(A)=\#\{ \text{limit-circle endpoints of } J\}.
	\]

	Assume now that $n_+(A)=n_-(A)=2$, and let $a$ and $b$ be the two limit-circle endpoints of $J$. 
	Moreover, let $\phi_1,\phi_2\in \dmax(A)$ be linearly independent modulo $\dmin(A)$ and normalized by 
	$
		[\phi_1,\phi_2](a)=[\phi_1,\phi_2](b)=1.
	$
	Then, $B$ is a self-adjoint extension of $A$ over $L^2(J,w(x)dx)$ if and only if 
	$Bu=A^*u$, for any $u\in\dom B$, and one of the following holds
	\begin{enumerate}
		\item \emph{Disjoint dynamics:} 
			there exists $c_+,c_-\in(-\infty,+\infty]$ such that $u\in\dom B$ if and only if
			\begin{gather*}
				[u,\phi_1](0^+)=c_{+}[u,\phi_2](a)\quad \text{and} \quad
				[u,\phi_1](0^-)=d_{+}[u,\phi_2](b).
			\end{gather*}
		\item \emph{Mixed dynamics:} 
			there exist $K\in SL_2(\R)$ such that $u\in\dom B$ if and only if
			\[
				U(bu)= K \,U(a), \qquad\text{ for } 
					U(x)= \left( \begin{array}{c} {[u,\phi_1](x)}\\ {[u,\phi_2](x)}\\ \end{array} \right).
			\]
	\end{enumerate}
\end{thm}

\begin{rmk}\label{rmk:dmax}
Let $\phi_1^a$ and $\phi_2^a$ be, respectively, the functions $\phi_1$ and $\phi_2$ of the above theorem, multiplied by a cutoff function $\eta:\overline J\to[0,1]$  supported in a (right or left) neighborhood of $a$ in $J$ and such that $\eta(a)=1$ and $\eta'(a)=0$.
Let $\phi_1^b$ and $\phi_2^b$ be defined analogously.
Then, from \eqref{eq:patching}, follows that we can write
\begin{equation}\label{eq:dmax2}
	\dmax(A)=\dmin(A)+\spn\{\phi_1^a,\phi_1^b,\phi_2^a,\phi_2^b\}.
\end{equation}
\end{rmk}

The following lemma classifies the end-points of $\R\setminus\{0\}$ with respect to the Fourier components of $\delcomp$.

\begin{lem} \label{lem:limitpoint}
    Consider the Sturm-Liouville operator $\widehat\Delta_k$ on $\R\setminus\{0\}$.
    Then, for any $k\in \Z$ the endpoints $+\infty$ and $-\infty$ are limit-point.
    On the other hand, regarding $0^+$ and $0^-$ the following holds.
  	\begin{enumerate}
  	    \item If $\alpha\le -3$ or if $\alpha\ge 1$, then they are limit-point for any $k\in\Z$;
  	    \item if $-3<\alpha\le -1$, then they are limit-circle if $k=0$ and limit-point otherwise;
  	    \item if $-1<\alpha<1$, then they are limit-circle for any $k\in\Z$.
  	\end{enumerate}
\end{lem}

Before the proof, we observe that, since $[u,v](d)=0$ for any limit-point end-point $d$, by the Patching Lemma \cite[Lemma~10.4.1]{zettl}  and \cite[Lemma~13.3.1]{zettl}, Lemma~\ref{lem:limitpoint} gives the following characterization of the minimal domain of $\widehat \Delta_k$,
\begin{equation} \label{eq:patching}
	\dmin(\widehat \Delta_k)=\left\{ u\in\dmax(\widehat \Delta_k)\mid 
			[u,v](0^+)=[u,v](0^-) = 0 \text{ for all } v\in\dmax(\widehat \Delta_k) \right\}.
\end{equation}

\begin{proof}[Proof of Lemma~\ref{lem:limitpoint}]
	By symmetry with respect to the origin of $\widehat\Delta_k$, it suffices to check only $0^+$ and $+\infty$.

	Let $k=0$, then for $\alpha\neq -1$ the equation $\widehat \Delta_0 u = u''-(\alpha /x) u' = 0$ has solutions $u_1(x)=1$ and $u_2(x)=x^{1+\alpha}$.
	Clearly, $u_1$ and $u_2$ are both in $L^2((0,1),|x|^{-\alpha}dx)$, i.e., $0^+$ is limit-circle, if and only if $\alpha\in (-3,1)$.
	On the other hand, $u_1$ and $u_2$ are never in $L^2((1,+\infty),|x|^{-\alpha}dx)$ simultaneously, and hence $+\infty$ is always limit-point.
	If $\alpha=-1$, the statement follows by the same argument applied to the solutions $u_1(x)=1$ and $u_2(x)=\log(x)$.
	
	Let now $k\neq 0$ and $\alpha\neq -1$. 
	Then $\widehat\Delta_k u=u''-(\alpha /x) u'-x^{2\alpha}k^2u=0$, $x>0$, has solutions $u_1(x)=\exp\left( \frac{kx^{1+\alpha}}{1+\alpha}\right)$ and $u_2(x)=\exp\left( -\frac{kx^{1+\alpha}}{1+\alpha}\right)$.
	If $\alpha>-1$, both $u_1$ and $u_2$ are bounded and nonzero near $x=0$, and either $u_1$ or $u_2$ has exponential growth as $x\rightarrow+\infty$. 
	Hence, in this case, $u_1,u_2\in L^2((0,1),|x|^{-\alpha})$ if and only if $\alpha< 1$, while $+\infty$ is always limit-point.
	On the other hand, if $\alpha<-1$, $u_1$ and $u_2$ are bounded away from zero as $x\rightarrow +\infty$ and one of them has exponential growth at $x=0$.
	Since the measure $|x|^{-\alpha}dx$ blows up at infinity, this implies that both $0^+$ and $+\infty$ are limit-point.
	Finally, the same holds for $\alpha=-1$, considering the solutions $u_1(x)=x^k$ and $u_2(x)=x^{-k}$.
\end{proof}

\subsection{Proofs of Theorem~\ref{thm:sa} and \ref{thm:sae}}
\label{sec:proof}

We are now able to classify the essential self-adjointness of the operator $\delcomp$.
 
\begin{proof}[Proof of Theorem~\ref{thm:sa}]
	Let $D\subset \cc M$ be the set of $\cc M$ functions which are finite linear combinations of products $u(x)v(\theta)$.
	Since $\ldue=L^2(\R\setminus\{0\},|x|^{-\alpha}dx)\otimes L^2(\Tu,d\theta)$, the set $D$ is dense in $\ldue$ and hence, by Proposition~\ref{prop:fourier} the operator $\Delta|_D$ is essentially self adjoint if and only if so are all $\widehat \Delta_k|_{D\cap H_k}$. 
	Since $n_\pm(\Delta|_D) = n_\pm(\delcomp)$, this is equivalent to $\delcomp$ being essentially self-adjoint.

	To conclude, recall that by Theorem~\ref{thm:zettl} the operator $\widehat \Delta_k$ is not essentially self-adjoint on $\lduef$ if and only if it is in the limit-circle case at at least one of the four endpoints $-\infty, \, 0^-,\,0^+ $ and $+\infty$. 
	Hence applying Lemma~\ref{lem:limitpoint} is enough to complete the proof.
\end{proof}

Now we proceed to study the self-adjoint extensions of the first Fourier component, proving Theorem~\ref{thm:sae} through Theorem~\ref{thm:zettl} and Remark~\ref{rmk:dmax}.

\begin{proof}[Proof of Theorem~\ref{thm:sae}]
	We start by proving the statement on $\dmin(\delfz)$.
	The operator $\delfz$ is transformed by the unitary map $U_0:\lduef\to L^2(\R\setminus\{0\})$, $U_0v(x)=|x|^{-\nicefrac{\alpha}{2}} v(x)$, in 
	\[
		\LBN_0 = \partial_x^2 - \frac \alpha 2 \left( \frac \alpha 2 +1\right)\frac 1 {x^2}.
	\]
	By \cite{bessel} and \cite[Lemma~13.3.1]{zettl}, it holds that $\dmin(\LBN_0)$ is the closure of $\cc {\R\setminus\{0\}}$ in the norm of $H^2({\R\setminus\{0\}},dx)$, i.e., 
	\[
		\|u\|_{H^2({\R\setminus\{0\}},dx)}=\|u\|_{L^2({\R\setminus\{0\}},dx)}+\|\partial_x u\|_{L^2({\R\setminus\{0\}},dx)}+\|\partial_x^2 u\|_{L^2({\R\setminus\{0\}},dx)}.
	\]
	From this follows that $\dmin(\delfz)=U_0^{-1}\dmin(\LBN_0)$ is given by the closure of $\cc {\R\setminus\{0\}}$ in $W=U_0^{-1}H^2({\R\setminus\{0\}},dx)$, w.r.t. the induced norm 
    \begin{equation}
      \label{eq:induced}
      \begin{split}
        \|v\|_{W} &= \| U_0 v \|_{H^2(\R\setminus\{0\},dx)} \\
        &= \|v\|_{\lduef}+\big\|
        |x|^{\nicefrac{\alpha}{2}}\partial_x(|x|^{-\nicefrac{\alpha}{2}}v) \big\|_{\lduef}+\||x|^{\nicefrac{\alpha}{2}}\partial_x^2(|x|^{-\nicefrac{\alpha}{2}}v)\|_{\lduef} \\
      \end{split}
    \end{equation}

    To prove the statement, it suffices to show that on $C^\infty_c(\R\setminus\{0\})$ the induced norm \eqref{eq:induced} is equivalent to the norm of $\sobduef$, which is
    \begin{equation}
      \label{eq:normh2}
    \|v\|_{\sobduef} = \|v\|_{\lduef}+\big\| \partial_x \,v \big\|_{\lduef}+\|\delfz v\|_{\lduef}.
  \end{equation}
    
    To this aim, observe that 
	\begin{equation}
		\label{eq:deriv}
		\partial_x v(x)= |x|^{\nicefrac{\alpha}{2}}\partial_x(|x|^{-\nicefrac{\alpha}{2}}v) + \frac{\alpha} 2 \frac{v}{x},\qquad
		 \delfz v =|x|^{\nicefrac{\alpha}{2}}\partial^2_x\big(|x|^{-\nicefrac{\alpha}{2}} v \big) + \frac{\alpha}{2} \left( \frac{\alpha}{2} +1\right) \frac{v}{x^2}.
	\end{equation}
    Moreover, by a cutoff argument, it is clear that we can prove the bound separately for $v$ supported near the origin and away from it.

    Let $v\in\cc{\R\setminus\{0\}}$ be supported in $(-1,0)\cup(0,1)$.
    By \eqref{eq:deriv} and the fact that if $|x|\le 1$ then $|x|^{-1}, |x|^{-2}\ge 1$, it follows immediately that $\|v\|_{\sobduef}\le \|v\|_W$.
    In order to prove the opposite inequality, observe that $x^{-2}\ge x^{-1}$ and $v\in H^2_0((0,1),dx)\oplus H^2_0((-1,0),dx)$.
    Thus, by \cite[(3.5)]{bessel} we obtain
    \begin{multline}
    \left\|  v {x^{-1}}\right\|_{\lduef} + \left\| v {x^{-2}} \right\|_{\lduef} \le 2\|vx^{-2}\|_{\lduef} \\
= 2 \| vx^{-2-\nicefrac{\alpha}{2}}  \|_{L^2((0,1))} 
    \le 2C \|\partial_x (v x^{-\nicefrac{\alpha}{2}})\|_{H^2((0,1))} = 2C \|v\|_{W}.
  \end{multline}

  Finally, let $v\in \cc{\R\setminus\{0\}}$ be supported in $(1,+\infty)$ (the same argument will work also between $(-\infty,-1)$).
  In this case, $x^{-2}<x^{-1}<1$.
  Thus, by \eqref{eq:deriv}, \eqref{eq:induced}, \eqref{eq:normh2} and the triangular inequality, we get that for any $v\in C^\infty_c(\R\setminus\{0\})$ it holds
    \begin{multline*}
      \label{eq:boundrhs}
      \left|\|v\|_W - \|v\|_{\sobduef}\right|\\ \le C\bigg( \left\|  v {x^{-1}}\right\|_{\lduef} + \left\| v {x^{-2}} \right\|_{\lduef}  \bigg) \le 2C\|v \|_{\lduef},
    \end{multline*}
    for some constant $C>0$.
    Since $\|v\|_W$ and $\|v\|_{\sobduef}\ge \|v\|_{\lduef}$, this completes the proof of the first part of the theorem.

	We now proceed to the classification of the self-adjoint extensions of $\delfz$.
	For this purpose, recall the definition of $\phidpm$ and $\phinpm$ given in \eqref{def:phidn} and let
	\begin{equation*}
	    \phin(x)=\phinp(x)+\phinm(x),
	    \qquad
	    \phid(x)=\phidp(x)+\phidm(x).
	\end{equation*}
Observe that $\phid\in\lduef$ and that $\delfz \phid (x) = 0$ for any $x \notin (-2,-1)\cup(1,2)$. 
Since the function is smooth, this implies that $\phid\in\dmax(\delfz)$.
The same holds for $\phin$.
Moreover, a simple computation shows that $[\phidp,\phinp](0^+)=[\phidp,\phinp](0^-)=1$, and hence $\phin$ and $\phid$ satisfy the hypotheses of Theorem~\ref{thm:zettl}.
In particular, by Remark~\ref{rmk:dmax}, this implies that
\[
	\dmax(\delfz)=\dmin(\delfz)+\spn\{\phidp,\phinp,\phidm,\phinm\}.
\]

We claim that for any $u=u_0+u_D^+\phidp+u_N^+\phinp+u_D^-\phidm+u_N^-\phinm\in\dmax$ it holds
\begin{equation}\label{eq:phid}
	[u,\phin](0^+)=u_D^+,\qquad [u,\phid](0^+)=u_N^+,\qquad [u,\phin](0^-)=u_D^-,\qquad [u,\phin](0^-)=u_N^-.
\end{equation}
This, by Theorem~\ref{thm:zettl} will complete the classification of the self-adjoint extensions.
Observe that, \eqref{eq:patching} and the bilinearity of the Lagrange parentheses imply that $[u_0,\phin](0^\pm)=[u_0,\phid](0^\pm)=0$.
The claim then follows from the fact that
\begin{gather*}
	[\phidp,\phin](0^+) = [\phinp,\phid](0^+) =[\phidm,\phin](0^-) = [\phinm,\phid](0^-) = 1, \\
	[\phidm,\phin](0^+) = [\phinm,\phid](0^+) =[\phidp,\phin](0^-) = [\phinp,\phid](0^-) = 0.
\end{gather*}

To complete the proof, it remains only to identify the Friedrichs extension $\delfzf$.
Recall that such extension is always defined, and has domain 
\[
	\dom \delfzf = \{ u\in \sobfzero\mid \delfz u\in\lduef \}.
\]
Since if $\alpha\le -1$, $\phin\notin\sobf$, it is clear that the Friedrichs extension corresponds to the case where $u_N^+=u_N^-=0$, i.e., to $c_+=c_-=0$.
On the other hand, if $\alpha>-1$, since all the end-points are regular, by \cite[Corollary~10.20]{weyl} holds that the Friedrichs extension corresponds to the case where $u(0^\pm)=u_D^\pm=0$, i.e., to $c_+=c_-=+\infty$.
\end{proof}

\begin{rmk}\label{rmk:dirneu}
	If $u\in\dmax(\delfz)$, it holds
	\[
		u_D^+=[u,\phin](0^+) = \lim_{x\downarrow 0} \big( u(x) - x \,\partial_x u(x) \big) \qquad \text{and} \qquad  u_N^+=[u,\phid](0^+) = \lim_{x\downarrow 0} x^{-\alpha} \,\partial_x u(x). 
	\]
	This implies, in particular, that if $\alpha>-1$ then $u_D^+=u(0^+)$. 
	Indeed this holds if and only if the end-point $0^+$ is regular in the sense of Sturm-Liouville operators, see Definition~\ref{defn:regular}.
	Clearly the same computations hold at $0^-$.
\end{rmk}	

We conclude this section with a description of the maximal domain.

\begin{prop}\label{prop:dmax}
    For any $\alpha\in\R$, it holds that
    \[
    	\dmax(\delcomp)=
    	\begin{cases}
    		\sobdue = \sobduezero & \quad\text{if } \alpha\le -3 \text{ or } \alpha\ge 1,\\
    		\sobdue \oplus \spn\{ \phinp, \phinm \} & \quad\text{if } -3<\alpha\le-1,\\
    		\sobdue \supsetneqq \sobduezero & \quad\text{if } -1<\alpha <1.
    	\end{cases}
    \]
    Here we let, with abuse of notation, $\phinpm(x,y)=\phinpm(x)$.
\end{prop}

\begin{proof}
	Recall that, by definition, $\sobdue\subset \dmax(\delcomp)$.
	Moreover, if $\alpha\le-3$ or if $\alpha\ge1$, by Theorem~\ref{thm:sa} it holds $\dmax(\delcomp)=\dom {\deld}=\sobduezero\subset \sobdue$. This proves the first statement.

    On the other hand, by Remark~\ref{rmk:dmax}, if $\alpha\in(-3,-1]$, since $\delfk$ is essentially self-adjoint for any $k\neq0$ we can decompose the maximal domain as
    \begin{equation*}
    	\begin{split}
	    	\dmax(\delcomp) &= \dmax(\delfz) \oplus \left(\bigoplus_{k\in\Z\setminus\{0\}} \dom{\delfk} \right) 
	    \end{split}
    \end{equation*}
    Moreover, letting $\pi_0$ be the projection on the $k=0$ Fourier component and defining $(\pi_0^{-1} u_0)(x,\theta)=u_0(x)$ for any $u_0\in\lduef$, the previous decomposition and the fact that $\dmin(\delcomp)\subset \sobdue\subset \dmax(\delcomp)$ implies that
    \[
    	\begin{split}
			\dmax(\delcomp) &= \big\{ u = u_0+\pi_0^{-1}\tilde u \mid u_0\in\dmin(\delcomp), \, \tilde u\in \spn\{\phidp,\phinp,\phidm,\phinm\}  \big\}\\
    		&= \sobdue+\spn\{\phidp,\phinp,\phidm,\phinm\}.
    	\end{split}
    \]
    Here, in the last equality, we let $\phid(x,y)=\phid(x)$ and $\phin(x,y)=\phin(x)$.
    A simple computation shows that $\phid\in\sobf$ and $\phin\notin\sobf$.
   	Since $\delfz\phid=0$, it follows that $\phid\in \sobdue$, while $\phin\notin\sobdue$. 
   	This implies the statement.
    
    To complete the proof it suffices to prove that if $\alpha\in(-1,1)$ it holds $\dmax(\delcomp)\subset\sobdue$.
    In fact, the inequality $\sobdue\neq\sobduezero$ will then follow from the fact that $\delf$ is not the only self-adjoint extension of $\delcomp$.
    By Parseval identity, $\phi,\Delta\phi \in\ldue$ if and only $\phi_k,\delfk\phi_k\in\lduef$ for any $k\in\Z$ and thus the statement is equivalent to $\dmax(\delfk)\subset\sobduef$ for any $k\in\Z$.
    Let $u\in\dmax(\delfk)$.
    Since $\lim_{x\rightarrow 0^\pm} x^{-\alpha}\partial_x u(x)=[u,\phid](0^\pm)$, this limit exists and is finite.
    Moreover, since $\pm\infty$ are limit-point, it holds $\lim_{x\rightarrow \pm\infty} x^{-\alpha}\partial_x u(x)=[u,\phid](\pm\infty)=0$.
    Hence, $x^{-\alpha}\partial_x u$ is square integrable near $0$ and at infinity, and from the characterization \eqref{eq:dmaxsturm} follows that $\delfk u\in\lduef$. 
    This proves that $u\in\sobduef$ and thus the proposition.
\end{proof}

 \section{Bilinear forms}
 \label{sec:bilinear}

 \subsection{Preliminaries}
 This introductory section is based on \cite{fukushima}.
 Let $\hilb$ be an Hilbert space with scalar product $(\cdot,\cdot)_\hilb$.
 A non-negative symmetric bilinear form densely defined on $\hilb$, henceforth called only a \em symmetric form on \em $\hilb$, is a map $\dex:\dom \dex\times\dom \dex \to \R$ such that $\dom \dex$ is dense in $\hilb$ and $\dex$ is bilinear, symmetric, and non-negative (i.e., $\dex(u,u)\ge 0$ for any $u\in \dom\dex$).
 A symmetric form is \em closed \em if $\dom \dex$ is a complete Hilbert space with respect to the scalar product
 \begin{equation}
 	\label{eq:scalar}
 	(u,v)_\dex = (u,v)_{\hilb} + \dex(u,v), \qquad u,v\in\dom \dex.
 \end{equation}
 To any densely defined non-positive definite self-adjoint operator $A$ it is possible to associate a symmetric form $\dex_A$ such that
 \begin{gather*}
 	\dex_A(u,v)=(-Au,v)\\
 	\dom A=\{ u \in \dom {\dex_A}\colon\: \exists v\in \hilb\text{ s.t. } \dex(u,\phi)=(v,\phi) \text{ for all }\phi\in \dom {\dex_A} \}.
 \end{gather*}
 Indeed, we have the following.

 \begin{thm}[\cite{kato, fukushima}]
 	\label{thm:kato}
 	Let $\hilb$ be an Hilbert space, then the map $A\mapsto \dex_A$ induces a one to one correspondence 
		\[
			A \text{ non-positive definite self-adjoint operator } \Longleftrightarrow \dex_A \text{ closed symmetric form}.
		\]
	In particular, the inverse correspondence can be characterized by $\dom A\subset \dom {\dex_A}$ and $\dex_A(u,v)=(-Au,v)$  for all  $u\in \dom A,\, v\in \dom {\dex_A}$.
 \end{thm}

 Consider now a second countable locally compact Hausdorff space $X$ with its Borel sigma algebra $\mathcal F$, and $m$ a Radon measure on $\mathcal F$ with full support.
 \begin{defn}
 	\label{def:markov}
 	A symmetric form $\dex$ on $L^2(X,m)$ is \em Markovian \em if for any $\eps>0$ there exists $\psi_\eps:\R\to\R$ such that $-\eps\le\psi_\eps\le 1+\eps$, $\psi_\eps(t)=t$ if $t\in [0,1]$,  $0\le\psi_\eps(t)-\psi_\eps(s)\le t-s$ whenever $s<t$ and
 	\[
 		u\in \dom{\dex} \Longrightarrow \psi_\eps(u)\in\dom{\dex} \quad\text{and}\quad \dex(\psi_\eps(u),\psi_\eps(u)) \le \dex(u,u).
 	\]
 	A closed Markovian symmetric form is a \em Dirichlet form\em.
 	
 	A semigroup $\{T_t\}_{t\ge 0}$ on $L^2(X,m)$ is \emph{Markovian} if 
 	\[
 		u\in L^2(X,m) \text{ s.t. } 0\le u \le 1 \quad m-\text{a.e.} \Longrightarrow 0\le T_t u \le 1 \quad m-\text{a.e. for any } t>0.
 	\]
 	
 	A non-positive self-adjoint operator is \emph{Markovian} if it generates a Markovian semigroup.
 \end{defn}
 
 When the form is closed, the Markov property can be simplified, as per the following Theorem. 
 For any $u:X\to \R$ let $u_\sharp=\min\{1,\max\{u,0\}\}$.

 \begin{thm}[Theorem 1.4.1 of \cite{fukushima}]
     \label{thm:normcontr}
    The closed symmetric form $\dex$ is Markovian if and only if
    \[
 		u\in \dom{\dex} \Longrightarrow u_\sharp\in\dom{\dex} \text{ and } \dex(u_\sharp,u_\sharp) \le \dex(u,u).
    \]
 \end{thm}

 Since any function of $L^\infty(X,m)$ is approximable by functions in $L^2(X,m)$, the Markov property allows to extend the definition of $\{T_t\}_{t\ge 0}$ to $L^\infty(X,m)$, and moreover implies that it is a contraction semigroup on this space.
 When $\{T_t\}_{t\ge 0}$ is the evolution semigroup of the heat equation, the Markov property can be seen as a physical admissibility condition.
 Namely, it assures that when starting from an initial datum $u$ representing a temperature distribution (i.e., a positive and bounded function) the solution $T_t u$ remains a temperature distribution at each time, and, moreover, that the heat does not concentrate.

 The following theorem extends the one-to-one correspondence given in Theorems~\ref{thm:stone} and \ref{thm:kato} to the Markovian setting.
 \begin{thm}[\cite{fukushima}]
 	Let $A$ be a non-positive self-adjoint operator on $L^2(X,m)$. The following are equivalents
 	\begin{enumerate}
 		\item $A$ is a Markovian operator;
 		\item $\dex_A$ is a Dirichlet form;
 		\item $\{e^{tA}\}_{t\ge 0}$ is a Markovian semigroup.
 	\end{enumerate}
 \end{thm}

 Given a non-positive symmetric operator $A$ we can always define the closable symmetric form
 \begin{equation*}
    \dex(u,v)= (-Au,v), \quad
    \dom \dex=\dom A.
 \end{equation*}
 The Friedrichs extension $A_F$ of $A$ is then defined as the self-adjoint operator associated via Theorem~\ref{thm:kato} to the closure $\dex_0$ of this form.
 Namely, $\dom {\dex_0}$ is the closure of $\dom A$ with respect to the scalar product \eqref{eq:scalar}, and $\dex_0(u,v)=\lim_{n\rightarrow+\infty} \dex(u_n,v_n)$ for $u_n\rightarrow u$ and $v_n\rightarrow v$ w.r.t. $(\cdot,\cdot)_\dex$.
 It is a well-known fact that the Friedrichs extension of a Markovian operator is always a Dirichlet form (see, e.g., \cite[Theorem 3.1.1]{fukushima}).

A Dirichlet form $\dex$ is said to be \emph{regular} on $X$ if $\dom \dex\cap C_c(X)$ is both dense in $\dom{\dex}$ w.r.t. the scalar product \eqref{eq:scalar} and dense in $C_c(X)$ w.r.t. the $L^\infty(X)$ norm.
To any regular Dirichlet form $\dex_A$ it is possible to associate a Markov process $\{X_t\}_{t\ge 0}$ which is generated by $A$ (indeed they are in one-to-one correspondence to a particular class of Markov processes, the so-called Hunt processes, see \cite{fukushima} for the details).

If its associated Dirichlet form is regular, by Definitions~\ref{def:stoc} and \ref{def:recurr}, a Markovian operator is said \emph{stochastically complete} if its associated Markov process has almost surely infinite lifespan, and \emph{recurrent} if it intersects any subset of $X$ with positive measure an infinite number of times.
If it is not stochastically complete, an operator is called \emph{explosive}.
Observe that recurrence is a stronger property than stochastic completeness,  but the two notions coincide when $m(X)<+\infty$, \cite[Section~2.11]{posilicano}.

We will need the following characterizations.
 \begin{thm}[Theorem 1.6.6 in \cite{fukushima}]
 	\label{thm:fukushima}
 	A Dirichlet form $\dex$ is stochastically complete if and only if there exists a sequence $\{u_n\}\subset \dom{\dex}$ satisfying
 	\[
 		0\le u_n\le 1,\quad \lim_{n\rightarrow+\infty} u_n =1 \quad m-\text{a.e.},
 	\]
 	such that
 	\[
 		\dex(u_n,v)\rightarrow 0 \quad \text{for any }v\in \dom\dex\cap L^1(X,m).
 	\]
\end{thm}

We let the \emph{extended domain} $\dom\dex_e$ of a Dirichlet form $\dex$ to be the family of functions $u\in L^\infty(X,m)$ such that there exists $\{u_n\}_{n\in\N}\subset\dom\dex$, Cauchy sequence w.r.t. the scalar product \eqref{eq:scalar},  such that $u_n\longrightarrow u$ $m$-a.e. . 
The Dirichlet form $\dex$ can be extended to $\dom\dex_e$ as a non-negative definite symmetric bilinear form, by $\dex(u,u)=\lim_{n\rightarrow+\infty}\dex(u_n,u_n)$.

\begin{thm}[Theorems 1.6.3 and 1.6.5 in \cite{fukushima}]
	\label{thm:fukushimarecur}
	Let $\dex$ be a Dirichlet form. The following are equivalent.
	\begin{enumerate}
		\item $\dex$ is recurrent;
		\item there exists a sequence $\{u_n\}\subset \dom{\dex}$ satisfying
 	\[
 		0\le u_n\le 1,\quad \lim_{n\rightarrow+\infty} u_n =1 \quad m-\text{a.e.},
 	\]
 	such that
 	\[
 		\dex(u_n,v)\rightarrow 0 \quad \text{for any }v\in \dom\dex_e.
 	\]
		\item $1\in\dom\dex_e$  and $\dex(1,1)=0$, i.e., there exists a sequence $\{u_n\}\subset \dom{\dex}$ such that $\lim_{n\rightarrow +\infty} u_n=1 \quad m-\text{a.e.}$ and $\dex(u_n,u_n)\rightarrow 0$.
	\end{enumerate}
 \end{thm}

 We conclude this preliminary part, by introducing a notion of restriction of closed forms associated to self-adjoint extensions of $\delcomp$.
 \begin{defn}
 	\label{def:neumann}
    Given a self-adjoint extension $A$ of $\delcomp$ and an open set $U\subset M$, we let the \emph{Neumann restriction} $\dex_A|_U$ of $\dex_A$ to be the form associated with the self-adjoint operator $A|_U$ on $L^2(U,\misura)$, obtained by putting Neumann boundary conditions on $\partial_U$.
 \end{defn}
 In particular, by Theorem~\ref{thm:kato} and an integration by parts, it follows that $\dom{\dex_A|_U}=\{ u|_U\mid u \in \dom {\dex_A}\}$.

 \subsection{Markovian extensions of $\delcomp$}
 \label{sec:markext}
 The bilinear form associated with $\delcomp$ is 
 \begin{equation*}
    \dex(u,v)= \int_{M_\alpha} \metr(\grad u, \grad v) \, \misura= \int_{M_\alpha} \left(\partial_x u\, \partial_x v + |x|^{2\alpha}\partial_\theta u\, \partial_\theta v\right)\,\misura, \quad
    \dom\dex=\cc M.
 \end{equation*}
 By \cite[Example 1.2.1]{fukushima}, $\dex$ is a Markovian form.
 The Friederichs extension is then associated with the form
 \begin{gather*}
    \dexf(u,v)=  \int_{M} \big(\partial_x u\, \partial_x v + |x|^{2\alpha}\partial_\theta u\, \partial_\theta v\big)\,\misura,\qquad
    \dom \dexf = \sobzero,
 \end{gather*}
 where the derivatives are taken in the sense of Schwartz distributions.
 By its very definition, and the fact that $\dom \dexf\cap C^\infty_c(M)=C^\infty_c(M)$, follows that $\dexf$ is always a regular Dirichlet form on $M$ (equivalently, on $M^+$ or on $M^-$).
 Its associated Markov process is absorbed by the singularity.

 The following Lemma will be crucial to study the properties of the Friederichs extension. 
 Let $M_0=(-1,1)\times \Tu$, $M_\infty=(1,+\infty)\times\Tu$ and recall the notion of Neumann restriction given in Definition~\ref{def:neumann}.
 
  \begin{lem}
  	\label{lem:sobzero}
    If $\alpha\le-1$, it holds that $1\in\dom{\dexf|_{M_0}}$.
   	Moreover, $1\notin \dom{\dexf|_{M_0}}_e$ if $\alpha>-1$ and $1\in \dom{\dexf|_{M_\infty}}_e$ if and only if $\alpha\ge -1$.
  \end{lem}

  \begin{proof}
  	To ease the notation, we let $\fdexfk$ to be the Dirichlet form associated to the Friederichs extension of $\delfk$.
  	In particular, for $k=0$,
  	\[
  		\fdexf(u,v)=\int_{\R\setminus\{0\}} \partial_x u\,\partial_x v\,|x|^{-\alpha}dx,\qquad \dom{\fdexf}=\sobfzero.
  	\]
	Let $\pi_k:\ldue\to H_k=\lduef$ be the projection on the $k$-th Fourier component.  	
	Then, from the rotational invariance of $\dom\dexf$ follows that
	\[
		\dom{\dexf}=\bigoplus_{k\in\Z} \dom{\fdexfk},\qquad \dexf(u,v)=\sum_{k\in\Z} \fdexfk(\pi_k u,\pi_k v).
	\]
	In particular, since $\pi_0 1=1$ and $\pi_k1=0$ for $k\neq 0$, follows that $1\in \dom{\dexf|_{M_0}}$ (resp. $1\in \dom{\dexf|_{M_\infty}}_e$) if and only if $1\in \dom{\fdexf|_{(0,1)}}$ (resp. $1\in \dom{\fdexf|_{(1,+\infty)}}_e$).
	Here, with abuse of notation, we denoted as $1$ both the functions $1:M\to \{1\}$ and $1:\R\to\{1\}$.
	Thus, to complete the proof of the lemma, it suffices to prove that $1\in\dom{\fdexf|_{(0,1)}}$ if $\alpha\le-1$, that $1\notin\dom{\fdexf|_{(0,1)}}_e$ if $\alpha\ge-1$ and that $1\in \dom{\fdexf|_{(1,+\infty)}}_e$ if and only if $\alpha\ge -1$.

    For any $0<r<R<+\infty$, let $f_{r,R}^\alpha$ be the only solution to the Cauchy problem
    \[
    	\begin{cases}
    	    \delfz f = 0,\\
    	    f(r)=1,\qquad f(R)=0.
    	\end{cases}
    \]
    Namely,
    \[    	
    f^\alpha_{r,R}(x)=
    	\begin{cases}
    	    \dfrac{R^{1+\alpha} - x^{1+\alpha}}{R^{1+\alpha}-r^{1+\alpha}}&\qquad\text{if } \alpha\neq -1,\\
    	    \\
    	    \dfrac {\log\left( \frac R x\right)}{\log\left( \frac R r\right)}&\qquad\text{if } \alpha= -1.
    	\end{cases}
   	\]
   	Then, the $0$-equilibrium potential (see \cite{fukushima} and Remark~\ref{rmk:equilibriumpotential}) of $[0,r]$ in $[0,R]$, is given by
     \begin{equation}
     	\label{eq:potential}
    	u_{r,R}(x)=
    	\begin{cases}
			    1 & \qquad \text{if } 0\le x \le r,\\
			    f_{r,R}^\alpha(x)& \qquad \text{if } r< x \le R,\\
			    0 & \qquad \text{if } x > R.\\    	
    	\end{cases}
   	\end{equation}
   	It is a well-known fact that $u_{r,R}$ is the minimizer for the capacity of $[0,r]$ in $[0,R)$. 
   	Namely, for any locally Lipschitz function $v$ with compact support contained in $[0,R]$ and such that $v(x)=1$ for any $0<x<r$, it holds
   	\begin{equation}
   		\label{eq:capacity}
   		\int_0^{+\infty} |\partial_x u_{r,R}|^2 x^{-\alpha}\,dx\le\int_0^{+\infty} |\partial_x v|^2 x^{-\alpha}\,dx
   	\end{equation}
    Since it is compactly supported on $[0,+\infty)$ and locally Lipschitz, it follows that $u_{r,R}\in \dom{\fdexf|_{(1,+\infty)}}$ and $1-u_{r,R}\in \dom{\fdexf|_{(0,1)}}$ for any $0<r<R<+\infty$.
        
    Consider now $\alpha\ge -1$, and let us prove that $1\in \dom{\fdexf|_{(1,+\infty)}}_e$.
    To this aim, it suffices to show that there exists a sequence $\{u_n\}_{n\in\N}\subset \dom{\fdexf|_{(1,+\infty)}}=\{u|_{(1,+\infty)}\mid u\in H^1((0,+\infty),x^{-\alpha}dx)\}$ such that $u_n\longrightarrow 1$ a.e. and $\fdexf|_{(1,+\infty)}$.
    Let 
    \[
    	u_n=
    	\begin{cases}
    	    u_{n,{2n}}\qquad & \text{ if }\alpha\neq-1,\\
    	    u_{n,n^2}\qquad & \text{ if }\alpha=-1.\\
    	\end{cases}
    \]
    It is clear that $u_n\longrightarrow 1$ a.e., moreover, a simple computation shows that
     \[
    	\fdexf|_{(1,+\infty)}(u_n,u_n)=\int_1^{+\infty} |\partial_x u_n|^2\,x^{-\alpha}\,dx =  
    	\begin{cases}
    		\frac{1+\alpha} {2^{1+\alpha}-1 }\,n^{-(1+\alpha)}\qquad&\text{if }\alpha\neq-1,\\
    		\frac 1 {\log\left(n\right)} \qquad&\text{if }\alpha=-1.
    	\end{cases}
    \]
    Hence $\fdexf|_{(1,+\infty)}\longrightarrow 0$ if $\alpha\ge -1$, proving that $1\in\dom{\fdexf|_{(1,+\infty)}}_e$.
    
   	We now prove that $1\in \dom{\fdexf|_{(0,1)}}$ if $\alpha\le -1$.
    Consider the following sequence in $H^1((0,1),x^{-\alpha}dx)$,
    \[
    	u_n= 
    	\begin{cases}
    	    u_{\nicefrac{1}{2n},\nicefrac{1}{n}}\qquad & \text{ if }\alpha\neq-1,\\
    	    u_{\nicefrac{1}{n^2},\nicefrac{1}{n}}\qquad & \text{ if }\alpha=-1.\\
    	\end{cases}
    \]
    A direct computation of $\int_0^1|\partial_x u_n|^2x^{-\alpha}dx$, the fact that $\supp u_n\subset [0,1/n]$ and $0\le u_n\le 1$, prove that $u_n\longrightarrow 0$ in $H^1((0,1),x^{-\alpha}dx)$.
    Since $1-u_n\in \dom{\fdexf|_{(0,1)}}$, which is closed, this proves that $1-u_n\longrightarrow 1$ in $\dom{\fdexf|_{(0,1)}}$, and hence the claim.

    To complete the proof, it remains to show that $1\notin\dom{\fdexf|_{(1,+\infty)}}_e$ if $\alpha<-1$.
    The same argument can be then used to prove that $1\notin\dom{\fdexf|_{(0,1)}}_e$ if $\alpha>-1$.
    We proceed by contradiction, assuming that there exists a sequence $\{v_n\}_{n\in\N}\subset\dom{\fdexf|_{(1,+\infty)}}$ such that $v_n\longrightarrow 1$ a.e. and  $\fdexf|_{(1,+\infty)}(v_n,v_n)\longrightarrow 0$.
    Since the form $\fdexf|_{(1,+\infty)}$ is regular on $[1,+\infty)$, we can take $v_n\in C^\infty_c([1,+\infty))$.
    Moreover, we can assume that $v_n(1)=1$ for any $n\in\N$.
    In fact, if this is not the case, it suffices to consider the sequence $\tilde v_n(x)=v_n(x)/v_n(1)$.
    Let $R_n>0$ be such that $\bigcup_{m\le n}\supp v_m\subset [1,R_n]$.
    Moreover, extend $v_n$ to $1$ on $(0,1)$, so that $\fdexf|_{(1,+\infty)}(v_n,v_n)=\int_0^{+\infty}|\partial_x v_n|^2x^{-\alpha}dx$.
    Since the same holds for $u_{1,R_n}$, by \eqref{eq:capacity}, the fact that $R_n\longrightarrow +\infty$ and $\alpha<-1$, we get
    \[
    	\lim_{n\rightarrow+\infty}\fdexf|_{(1,+\infty)}(v_n,v_n)\ge \lim_{n\rightarrow+\infty}\fdexf|_{(1,+\infty)}(u_{1,R_n},u_{1,R_n})=\lim_{n\rightarrow+\infty}\frac{1+\alpha}{R_n^{1+\alpha}-1} =-({1+\alpha})>0.
    \]
    This contradicts the fact that $\fdexf|_{(1,+\infty)}(v_n,v_n)\longrightarrow 0$, completing the proof.
  \end{proof}

  \begin{rmk}
  	\label{rmk:equilibriumpotential}
  	The $0$-equilibrium potential defined in \eqref{eq:potential} admits a probabilistic interpretation, \cite{GrigorAnalitic}.
  	Namely, it is the probability that the Markov process associated with $\delfz$ and starting from $x$, exits the first time from the interval $\{r<x<R\}$ through the inner boundary $\{x=r\}$.
  \end{rmk}

  It is possible to define a semi-order on the set of the Markovian extensions of $\delcomp$ as follows. 
 Given two Markovian extensions $A$ and $B$, we say that $A\subset B$ if $\dom{\dex_A}\subset \dom{\dex_B}$ and $\dex_A(u,u)\ge \dex_B(u,u)$ for any $u\in \dom{\dex_A}$.
 With respect to this semi-order, the Friederichs extension is the minimal Markovian extension.
 Let $\delm$ be the maximal Markovian extension (see \cite{fukushima}). 
 This extension is associated with the Dirichlet form $\dexm$ defined by
 \begin{gather*}
    \dexm(u,v)=  \int_{M} \big(\partial_x u\, \partial_x v + |x|^{2\alpha}\partial_\theta u\, \partial_\theta v\big)\,\misura,\\
    \dom \dexm = \{ u\in\ldue\mid \dexm(u,u)<+\infty\} = \sob,
 \end{gather*}
 where the derivatives are taken in the sense of Schwartz distributions.
 We remark that $\dexm$ is a regular Dirichlet form on $\overline {M^+} = M_\alpha\setminus M^-$ and $\overline {M^-} = M_\alpha\setminus M^+$ (see, e.g., \cite[Lemma~3.3.3]{fukushima}). 
 Its associated Markov process is reflected by the singularity.

 When $\delcomp$ has only one Markovian extension, i.e., whenever $\delf=\delm$, we say that it is \emph{Markov unique}.
 Clearly, if $\delcomp$ is essentially self-adjoint, it is also Markov unique.
 The next proposition shows that essential self-adjointness is a strictly stronger property than Markov uniqueness.
 
\begin{prop}\label{prop:onlyfried}
    The operator $\delcomp$ is Markov unique if and only if $\alpha\notin(-1,1)$.
\end{prop}

\begin{proof}
	As observed above, the statement is an immediate consequence of Theorem~\ref{thm:sa} for $\alpha\le-3$ and $\alpha\ge 1$.
	If $\alpha\in(-3,-1]$, since by Theorem~\ref{thm:sa} all $\delfk$ for $k\neq 0$ are essentially self-adjoint, it holds that $\delm=\widehat A_0 \oplus(\bigoplus_{k\in\N}\delfk)$ for some self-adjoint extension $\widehat A_0$ of $\delfz$.
	Recall the definition of $\phidpm$ and $\phinpm$ given in \eqref{def:phidn} and with abuse of notation let $\phidpm(x,\theta)=\phidpm(x)$ and $\phinpm(x,\theta)=\phinpm(x)$.
    Since $\dexm(\phin^\pm,\phin^\pm)=+\infty$ if and only if $\alpha\le-1$,
    we get that $\phinp,\phinm\notin\dom\dexm\supset\dom {\delm}$ if $\alpha\le-1$.
    Hence, by Theorem~\ref{thm:sae}, it holds that $\widehat A_0=(\delfz)_F$ and hence that $\delm=\delf$.

    On the other hand, if $\alpha\in(-1,1)$, the result follows from Lemma~\ref{lem:sobzero}. 
    In fact, it implies that $\phid\notin\sobzero=\dom\dexf$ but, since $\dexm(\phid,\phid)<+\infty$, we have that $\phid\in \dom \dexm$.
    This proves that $\delf\subsetneqq \delm$.
\end{proof}

 By the previous result, when $\alpha\in (-1,1)$ it makes sense to consider the \bridge extension, associated to the operator $\deltabridge$ and the form $\dexb$, defined by
\begin{gather*}
    \dexb(u,v)=  \int_{M_\alpha} \big(\partial_x u\, \partial_x v + |x|^{2\alpha}\partial_\theta u\, \partial_\theta v\big)\,\misura,\\
    \dom \dexb = \{ u\in \sob \mid u(0^+,\theta)=u(0^-,\theta) \text{ for a.e. }\theta\in\Tu\}.
 \end{gather*}
From Theorem~\ref{thm:normcontr} and the fact that $\dexb=\dexm|_{\dom \dexb}$ follows immediately that $\dexb$ is a Dirichlet form, and hence $\delf\subset\delb\subset\delm$.
Moreover, due to the regularity of $\dexm$ and the symmetry of the boundary conditions appearing in $\dom\dexb$, follows that $\dexb$ is regular on the whole $M_\alpha$.
Its associated Markov process can cross, with continuous trajectories, the singularity.

 We conclude this section by specifying the domains of the Markovian self-adjoint extensions associated with $\dexf$, $\dexm$ and, when it is defined, $\dexb$.

 \begin{prop}
 	\label{prop:domains}
    It holds that $\dom {\delf} = \sobduezero$, while
    \begin{equation*}
	\dom {\delm} = \{ u\in\sob \mid (\Delta u,v)=(\nabla u,\nabla v) \text{ for any } v\in \sob\}.
	\end{equation*}
	Moreover, if $\alpha\in(-1,1)$, the domain of $\delb$ is 
	\[	
	\dom {\delc} = \{ \sobduec \mid u(0^+,\cdot)=u(0^-,\cdot),\, \lim_{x\rightarrow 0^+} |x|^{-\alpha}\partial_x u(x,\cdot)=\lim_{x\rightarrow 0^-} |x|^{-\alpha}\partial_x u(x,\cdot) \text{ for a.e. }\theta\in\Tu\}. 
	\]
 \end{prop}

 \begin{proof}
 	In view of Theorem~\ref{thm:kato}, to prove that $A$ is the operator associated with $\dex_A$ it suffices to prove that $\dom A\subset \dom {\dex_A}$ and that $\dex_A(u,v)=(-Au,v)$ for any $u\in\dom A$ and $v\in \dom{\dex_A}$.
 	The requirement on the domain is satisfied by definition in all three cases.
 	We proceed to prove the second fact.

 	\emph{Friedrichs extension.} By integration by parts it follows that $\dexf(u,v)=(-\delf u,v)$ for any $u,v\in \cc M$, and this equality can be extended to $u\in\sobduezero=\dom{\delf}$ and $v\in \sobzero=\dom\dexf$.  

 	\emph{Neumann extension.} The property that $\dexm(u,v)=(-\delm u,v)$ for any $u\in\dom \delm$ and $v\in \dom{\dexm}$ is contained in the definition.

 	\emph{Bridging extension.} 
 	By an integration by parts, it follows that
 	\[
 		\int_{\clm} \big(\partial_x u\, \partial_x v + x^{2\alpha}\partial_\theta u\, \partial_\theta v\big)\,\misura = (-\delb u,v) - \int_\Tu v |x|^{-\alpha} \partial_x u \big|_{x=0^-}^{0^+}\,d\theta = (-\delb u,v). 
 	\]
  \end{proof}

\subsection{Stochastic completeness and recurrence on $M_\alpha$}
 \label{sec:stoccompl}

We are interested in localizing the properties of stochastic completeness and recurrence of a Markovian self-adjoint extension $A$ of $\delcomp$.
Due to the already mentioned repulsing properties of Neumann boundary conditions, the natural way to operate is to consider the Neumann restriction introduced in Definition~\ref{def:neumann}.

Observe that, if $U\subset M$ is an open set such that $\bar U\cap (\{-\infty,0,+\infty\}\times\Tu)=\varnothing$, then the Neumann restriction $\dex_A|_U$ is always recurrent on $U$.
In fact, in this case, there exist two constants $0<C_1<C_2$ such that $C_1 dx\,d\theta\le \misura\le C_2dx\,d\theta$ on $U$ and clearly $1\in \dom{\dex_A|_U}=H^1(U,dx\,d\theta)$, that by Theorem~\ref{thm:fukushimarecur} implies the recurrence.
For this reason, we will concentrate only on the properties ``at $0$'' or ``at $\infty$''.

 \begin{defn}
 	\label{def:stocorigin}
    Given a Markovian extension $A$ of $\delcomp$, we say that it is \em stochastically complete at $0$ \em (resp. \em recurrent at $0$\em) if its Neumann restriction to $M_0=(-1,1)\times\Tu$, is stochastically complete (resp. recurrent).
    We say that $A$ is exploding at $0$ if it is not stochastically complete at $0$.
    Considering $M_\infty=(1,\infty)\times\Tu$, we define stochastic completeness, recurrence and explosiveness at $\infty$ in the same way.
 \end{defn}

 In order to justify this approach, we will need the following.

 \begin{prop}
 	\label{prop:unione}
    A Markovian extension $A$ of $\delcomp$ is  stochastically complete (resp. recurrent) if and only if it is stochastically complete (resp. recurrent) both at $0$ and at $\infty$. 
 \end{prop}

\begin{proof}
    Let $\{u_n\}_{n\in\N}\subset \dom{\dex_A}$ such that $u_n\rightarrow 1$ a.e. and $\dex_A(u_n,u_n)\rightarrow 0$.
    Since $\dom{\dex_A|_{M_0}}=\{u|_{M_0}\mid u\in\dom{\dex_A}\}$ and $\dom{\dex_A|_{M_\infty}}=\{u|_{M_\infty}\mid u\in\dom{\dex_A}\}$ follows that $\{u_n|_{M_0}\}_{n\in\N}\subset\dom{\dex_A|_{M_0}}$ and $\{u_n|_{M_\infty}\}_{n\in\N}\subset\dom{\dex_A|_{M_\infty}}$.
    Moreover, it is clear that $u_n|_{M_0},u_n|_{M_\infty}\rightarrow 1$ a.e. and $\dex_A|_{M_0}(u_n|_{M_0},u_n|_{M_0}),$ $ \dex_A|_{M_\infty}(u_n|_{M_\infty},u_n|_{M_\infty})\rightarrow 0$.
    By Theorem~\ref{thm:fukushimarecur}, this proves that if $\dex_A$ is recurrent it is recurrent also at $0$ and $\infty$.

    On the other hand, if $A|_{M_0}$ and $A|_{M_\infty}$ are recurrent, we can always choose the sequences $\{u_n\}_{n\in\N}\subset\dom{\dex_A|_{M_0}}$ and $\{v_n\}_{n\in\N}\subset\dom{\dex_A|_{M_\infty}}$ approximating $1$ such that they equal $1$ in a neighborhood $N$ of $\partial_{M_0}=\partial_{M_\infty}=(\{1\}\times\Tu)\cup(\{-1\}\times\Tu)$.
    In fact the constant function satisfies the Neumann boundary conditions we posed on $\partial M_0=\partial M_\infty$ for the operators associated with $\dex_A|_{M_0}$ and $\dex_A|_{M_\infty}$.
    Hence, by gluing $u_n$ and $v_n$ we get a sequence of functions in $\dom{\dex_A}$ approximating $1$.
    The same argument gives also the equivalence of the stochastic completeness, exploiting the characterization given in Theorem~\ref{thm:fukushima}.
\end{proof}

Before proceeding with the classification of the stochastic completeness and recurrence of $\delf$, $\delm$ and $\delb$, we need the following result.
For an operator acting on $\lduef$, the definition of stochastic completeness and recurrence at $0$ or at $\infty$ is given substituting $M_0$ and $M_\infty$ in Definition~\ref{def:stocorigin} with $(-1,1)$ and $(1,+\infty)$.

 \begin{prop} \label{prop:zero}
 	Let $A$ be a Markovian self-adjoint extension of $\delcomp$ and assume it decomposes as $A=\widehat A_0\oplus \tilde A$, where $\widehat A_0$ is a self-adjoint operator on $H_0$ and $\tilde A$ is a self-adjoint operator on $\bigoplus_{k\neq0}H_k$.
 	Then, $\widehat A_0$ is a Markovian self-adjoint extension of ${\widehat{\Delta}_0}$.
 	Moreover, $A$ is stochastically complete (resp. recurrent) at $0$ or at $\infty$ if and only if so is $\widehat A_0$.
 \end{prop}

 \begin{proof}
 	Let $\pi_k:\ldue\to H_k=\lduef$ be the projection on the $k$-th Fourier component. In particular, recall that $\pi_0 u = (2\pi)^{-1}\int_0^{2\pi}u(x,\theta)\,d\theta$.
	Let $u\in \dom {\widehat A_0} \subset L^2(\R,|x|^{-\alpha}dx)$ be such that $0\le u \le 1$. 
	Hence, posing $\tilde u(x,\theta)=u(x)$, due to the splitting of $A$ follows that $\tilde u\in\dom A$ and by the markovianity follows that $0\le A \tilde u\le 1$. 
	The first part of the statement is then proved by observing that, since $\pi_0 \tilde u=u$ and $\pi_k \tilde u=0$ for $k\neq 0$, we have $A \tilde u(x,\theta) = \widehat A_0 u(x)$ for any $(x,\theta)\in M$.

	We prove the second part of the statement only at $0$, since the arguments to treat the at $\infty$ case are analogous.
	First of all, we show that the stochastic completeness of $A$ and $\widehat A_0$ at $0$ are equivalent.
	If $1:M_0\to \R$ is the constant function, it holds that $\pi_0 1=1:(-1,1)\to \R$. 
	Moreover, due to the splitting of $A$, we have that $e^{tA}=e^{t\widehat A_0}\oplus e^{t\tilde A}$
	Hence, it follows that $e^{t A}1 = e^{t\widehat A_0} 1$.
	This, by Definition~\ref{def:stoc}, proves the claim.

	To prove the equivalence of the recurrences at $0$, we start by observing that $\dom {\dex_A}=\dom {\dex_{\widehat A_0}}\oplus \dom{\dex_{\tilde A}}$ and that 
	\begin{equation}
		\label{eq:dexfourier}
		\dex_A(u,v)=\dex_{\widehat A_0}(\pi_0 u,\pi_0 v) + \dex_{\tilde A}(\oplus_{k\neq 0} \pi_k u,\oplus_{k\neq 0} \pi_k v),\qquad \text{ for any }u,v\in\dom{\dex_A}
	\end{equation}
	In particular, since $\pi_0 1=1$ this implies that $\dex_A|_{M_0}(1,1)=\dex_{\widehat A_0}|_{(-1,1)}(1,1)$.
	By Theorem~\ref{thm:fukushimarecur}, this proves that if $\widehat A_0$ is recurrent at $0$, so is $A$.
	Assume now that $A|_{M_0}$ is recurrent.
	By Theorem~\ref{thm:fukushimarecur} there exists $\{u_n\}_{n\in\N}\subset \dom {\dex_A|_{M_0}}$ such that $0\le u_n\le 1$ a.e., $u_n\longrightarrow 1$ a.e. and $\dex_A|_{M_0}(u_n,v)\rightarrow 0$ for any $v$ in the extended domain $\dom{\dex_A|_{M_0}}_e$. 
	By dominated convergence, it follows that $\pi_0 u_n=(2\pi)^{-1}\int_0^{2\pi}u_n(\cdot,\theta)\,d\theta \rightarrow 1$ for a.e. $x\in(-1,1)$.
	For any $v\in \dom{\dex_{\widehat A_0}|_{(-1,1)}}_e$, let $\tilde v(x,\theta)=v(x)$.
	It is easy to see that $\tilde v\in \dom{\dex_A|_{M_0}}_e$
	Then, by applying \eqref{eq:dexfourier} we get
	\[
		\dex_{\widehat A_0}|_{(-1,1)}({\pi_0 u_n,v}) = \dex_A|_{M_0} (u_n,\tilde v)\longrightarrow 0,\qquad \text{ for any } v\in \dom{\dex_{\widehat A_0}|_{(-1,1)}}_e.
	\]
	Since $0\le \pi_0 u_n \le 1$, this proves that $\widehat A_0|_{(-1,1)}$ is recurrent
 \end{proof}

The following proposition answers the problem of stochastic completeness or recurrence of the Friedrichs extension. 

\begin{prop}\label{prop:fried}
 	Let $\Delta_F$ be the Friedrichs extension of $\delcomp$. 
 	Then, the following holds 
 	\begin{center}
		\begin{tabular}{c|c|c|}
			 & at $0$ & at $\infty$ \\ \hline
			$\alpha < -1$ & recurrent& stochastically complete  \\ \hline
			$\alpha = -1$ & recurrent & recurrent \\ \hline
			$\alpha > -1$ &  explosive & recurrent\\ \hline		
		\end{tabular}
	\end{center}
	In particular, $\delf$ is stochastically complete for $\alpha<-1$, recurrent for $\alpha=-1$ and explosive for $\alpha>-1$.
\end{prop}

\begin{proof}
	The part regarding the recurrence is a consequence of Lemma~\ref{lem:sobzero} and Theorem~\ref{thm:fukushimarecur}, while the last statement is a consequence of Proposition~\ref{prop:unione}.
	Thus, to complete the proof it suffices to prove that $\delf$ is stochastically complete at $+\infty$ if $\alpha<-1$ and not stochastically complete at $0$ if $\alpha>-1$.

	By Proposition~\ref{prop:zero} and the fact that $\delf=\oplus_{k\in\Z}(\delfk)_F$, we actually need to prove this fact only for $(\delfz)_F$.
	Moreover, since the Friederichs extension decouples the dynamics on the two sides of the singularity, we can work only on $(0,+\infty)$ instead that on $\R\setminus\{0\}$.
 	As in Lemma~\ref{lem:sobzero}, we let $\widehat\dex_0$ to be the Dirichlet form associated to the Friederichs extension of $\delfz$.
    
    We start by proving the explosion for $\alpha>-1$ on $(0,1)$.
    Let us proceed by contradiction and assume that $(\delfz)_F$ is stochastically complete on $(0,1)$.
    By Theorem~\ref{thm:fukushima}, there exists $u_n\in \dom{\fdexf|_{(0,1)}}$, $0\le u_n\le 1$, $u_n\longrightarrow 1$ a.e. and such that $\fdexf|_{(0,1)}(u_n,v)\longrightarrow 0$ for any $v\in\dom {\fdexf|_{(0,1)}}\cap L^1((0,1),x^{-\alpha}dx)$.
    Since $\fdexf|_{(0,1)}$ is regular on $(0,1]$, we can choose the sequence such that $u_n\in\cc{(0,1]}$.
    In particular $u_n(0)=\lim_{x\downarrow 0} u_n(x)=0$ for any $n$.
    Let us define, for any $0<R\le 1$,
    \[
    	v_R(x)=\lim_{r\downarrow 0} \big(1-u_{r,R}(x)\big)=
    	\begin{cases}
    		{x^{1+\alpha}}/{R^{1+\alpha}} & \qquad \text{if }0\le x<R,\\
    		1 & \qquad \text{if } 0\le x\ge R,
    	\end{cases}
    \]
    where $u_{r,R}$ is defined in \eqref{eq:potential}.
    Observe that, by the probabilistic interpretation of $u_{r,R}$ given in Remark~\ref{rmk:equilibriumpotential}, follows that $v_R(x)$ is the probability that the Markov process associated with $(\delfz)_F$ and starting from $x$ exits the interval $(0,R)$ before being absorbed by the singularity at $0$.
    A simple computation shows that $v_R\in \dom{\fdexf|_{(0,1)}}\cap L^1((0,1),x^{-\alpha}dx)$.
    Thus, by definition of $\{u_n\}_{n\in\N}$ and a direct computation we get
    \[
    	0=\lim_{n\rightarrow+\infty} \fdexf|_{(0,1)}(u_n,v_{R})= 
    	\frac{1+\alpha}{R^{1+\alpha}} \lim_{n\rightarrow+\infty}\int_0^{R} \partial_x u_n\,dx = 
    	\frac{1+\alpha}{R^{1+\alpha}}  \lim_{n\rightarrow+\infty} u_n(R).
    \]
    Hence, $u_n(R)\longrightarrow 0$ for any $0<R<1$, contradicting the fact that $u_n\longrightarrow 1$ a.e..

    To complete the proof, we need to show that if $\alpha<-1$, $(\delfz)_F$ is stochastically complete on $(1,+\infty)$.
    Since on $(1,+\infty)$ the metric is regular, we can complete it to a $C^{\infty}$ Riemannian metric on the whole interval $(0,+\infty)$.
    Then, the result follows by applying the characterization of stochastic completeness on model manifolds contained in \cite{grigoryan} and Theorem~\ref{prop:unione}.
\end{proof}

We are now in a position to prove Theorem~\ref{thm:markov}.

 \begin{proof}[Proof of Theorem~\ref{thm:markov}]
  	By Propositions~\ref{prop:onlyfried} and \ref{prop:fried}, we are left only to prove statement {$(iii)$-$(a)$} and the second part of $(iii)$-$(b)$, i.e., the stochastic completeness of $\delm$ and $\delb$ at $0$ when $\alpha\in(-1,1)$.

 	Statement $(iii)$-$(a)$ follows from \cite[Theorem~1.6.4]{fukushima}, since for $\alpha\in(-1,1)$ the Friederichs extension (which is the minimal extension of $\delcomp$) is recurrent at $\infty$. 
 	To complete the proof it suffices to observe that, for these values of $\alpha$, it holds that $1\in H^1(M_0,\misura)=\dom{\dexm|_{M_0}}$ and clearly $\dexm|_{M_0}(1,1)=0$.
 	By Theorem~\ref{thm:fukushimarecur}, this implies the recurrence of $\dexm$ at $0$.
 	The recurrence of $\dexb$ at $0$ follows analogously, observing that $1$ is also continuous on $\sing$ and hence it belongs to $\dom{\dexb|_{M_0}}$
 \end{proof}

\appendix
\section{Geometric interpretation}

In this appendix we prove Lemmata~\ref{lem:mcil} and \ref{lem:mcono}, and justify the geometric interpretation of Figure~\ref{fig:int-geo}.

\subsection{Topology of $M_\alpha$}
\label{sec:geom}

\begin{proof}[Proof of Lemma~\ref{lem:mcil}]
	By \eqref{eq:distcontr}, it is clear that $d:\mcil\times \mcil\to [0,+\infty)$ is symmetric, satisfies the triangular inequality and $d(q,q)=0$ for any $q\in \mcil$.
	Observe that the topology on $\mcil$ is induced by the distance $\distcil((x_1,\theta_1),(x_2,\theta_2))=|x_1-x_2|+|\theta_1-\theta_2|$.
	Here and henceforth, for any $\theta_1,\,\theta_2\in\Tu$ when writing $\theta_1-\theta_2$ we mean the non-negative number $\theta_1-\theta_2\mod 2\pi$.
	In order to complete the proof it suffices to show that for any $\{q_n\}_{n\in\N}\subset \mcil$ and $\bar q\in \mcil$ it holds
	\begin{equation}
		\label{eq:equivcil}
		d(q_n,\bar q)\longrightarrow 0 \text{ if and only if }\distcil(q_n,\bar q)\longrightarrow 0.
	\end{equation}
	In fact, this clearly implies that if $d(q_1,q_2)=0$ then $q_1=q_2$, proving that $d$ is a distance, and moreover proves that $d$ and $\distcil$ induce the same topology.	
	
	Assume that $d(q_n,\bar q)\rightarrow 0$ for some $\{q_n\}_{n\in\N}\subset\mcil$ and $\bar q=(\bar x,\bar\theta)\in \mcil$.
	In this case, for any $n\in\N$ there exists a control $u_n:[0,1]\to \R^2$ such that $\|u_n\|_{L^1([0,1],\R^2)}\rightarrow 0$ and that the associated trajectory $\gamma_n(\cdot)=(x_n(\cdot),\theta_n(\cdot))$ satisfies $\gamma_n(0)=q_n$ and $\gamma_n(1)=\bar q$.
	This implies that, for any $t\in[0,1]$
	\[
		|x_n(t)-\bar x|\le \int_0^t |u_1(t)|\,dt \le \|u_n\|_{L^1([0,1],\R^2)}\longrightarrow 0.
	\]
	Hence, $x_n(t)\longrightarrow \bar x$. 
	In particular, this implies that $|x_n(t)|\le \|u_n\|_{L^1([0,1],\R^2)} +|\bar x|$ for any $t\in[0,1]$, and hence
	\begin{multline*}
		|\theta_n(0)-\bar\theta|\le \int_0^1 |u_2(t)||x_n(t)|^\alpha\,dt\le\big(\|u_n\|_{L^1([0,1],\R^2)} +|\bar x|\big)^\alpha\int_0^1 |u_2(t)|\,dt \\
		\le \|u_n\|_{L^1([0,1],\R^2)}(\|u_n\|_{L^1([0,1],\R^2)} +|\bar x|)^\alpha \longrightarrow 0.
	\end{multline*}
	Here, when taking the limit, we exploited the fact that $\alpha\ge0$.
	Thus also $\theta_n(0)\longrightarrow\bar\theta$, and hence $q_n=(x_n(0),\theta_n(0))\longrightarrow (\bar x,\bar\theta)=\bar q$ w.r.t. $\distcil$.
	
	In order to complete the proof of \eqref{eq:equivcil}, we now assume that for some $q_n=(x_n,\theta_n)$ and $\bar q=(\bar x,\bar\theta)$ it holds $\distcil(q_n,\bar q)\longrightarrow 0$ and claim that $d(q_n,\bar q)\longrightarrow 0$.
	We start by considering the case $\bar q\notin \sing$, and w.l.o.g. we assume $\bar q\in M^+$.
	Since $M^+$ is open with respect to $\distcil$, we may assume that $q_n\in M^+$. 
	Consider now the controls
	\[
		u_n(t)=
		\begin{cases}
			2\,(\bar x - x_n) \, (1,0)&\qquad\text{if } 0\le t\le\frac 1 2,\\
		    2\,(\bar\theta-\theta_n)|\bar x|^{-\alpha}\,(0,1) &\qquad\text{if } \frac 1 2<t\le 1,\\
		\end{cases}
	\]
	A simple computation shows that each $u_n$ steers the system from $q_n$ to $\bar q$. 
	The claim then follows from
	\[
		d(q_n,\bar q)\le \|u_n\|_{L^1([0,1],\R^2)}\le |\bar x-x_n|+|\bar\theta-\theta_n||\bar x|^{-\alpha}\le (1+|\bar x|^{-\alpha} )\,\distcil(q_n,\bar q) \longrightarrow 0.
	\]
	Let now $\bar q\in\sing$ and observe that w.l.o.g. we can assume $q_n\notin \sing$ for any $n\in\N$.
	In fact, if this is not the case it suffices to consider $\tilde q_n=(x_n+1/n,\theta_n)\notin\sing$, observe that $d(q_n,\tilde q_n)\rightarrow 0$ and apply the triangular inequality.
	Then, we consider the following controls, steering the system from $q_n$ to $\bar q$,
	\[
		v_n(t)=
		\begin{cases}
			3\,\big((\bar\theta-\theta_n)^{\nicefrac{1}{2\alpha}}-x_n\big) \, (1,0)&\qquad\text{if } 0\le t\le\frac 1 3,\\
		    3\,(\bar\theta-\theta_n)^{\nicefrac{1}{2}}\,(0,1) &\qquad\text{if } \frac 1 3<t\le\frac 2 3,\\
		    -3\,(\bar\theta-\theta_n)^{\nicefrac{1}{2\alpha}}\,(1,0)&\qquad\text{if } \frac 2 3< t\le1.\\
		\end{cases}
	\]
	Since $\bar x=0$ and $\alpha\ge 0$, we have
	\[
		d(q_n,\bar q)\le \|v_n\|_{L^1([0,1],\R^2)}\le |(\theta_n-\bar\theta)^{\nicefrac{1}{2\alpha}}-x_n|+|\bar\theta-\theta_n|^{\nicefrac{1}{2}}+|\theta_n-\bar\theta|^{\nicefrac{1}{2\alpha}} \longrightarrow 0.
	\]
	This proves \eqref{eq:equivcil} and hence the lemma.
\end{proof}

\begin{proof}[Proof of Lemma~\ref{lem:mcono}]
	By \eqref{eq:distmetr}, it is clear that $d:\mcono\times \mcono\to [0,+\infty)$ is symmetric, satisfies the triangular inequality and $d(q,q)=0$ for any $q\in \mcono$.

	Observe that the topology on $\mcono$ is induced by the following metric
    \[
        \distcono((x_1,\theta_1),(x_2,\theta_2))=
        \begin{cases}
            |x_1-x_2|+|\theta_1-\theta_2| &\qquad\text{if } x_1x_2> 0,\\
            |x_1-x_2| &\qquad\text{if } x_1=0 \text{ or } x_2=0,\\
            |x_1-x_2|+|\theta_1|+|\theta_2| &\qquad\text{if } x_1x_2< 0.\\
        \end{cases}
    \]
	By symmetry, to show the equivalence of the topologies induced by $d$ and by $\distcono$, it is enough to show that the two distances are equivalent on $[0,+\infty)\times\Tu$.
	Moreover, since by definition of $\metr$ it is clear that $d(q_1,q_2)=0$ for any $q_1,q_2\in\sing$ and that $d$ is equivalent to the Euclidean metric on $(0,+\infty)\times\Tu$, we only have to show that for any $\{q_n\}\subset (0,+\infty)\times \Tu$, $q_n=(x_n,\theta_n)$, and $\bar q=(0,\bar\theta) \in\sing$, it holds that 
	\begin{equation}
		\label{eq:equivcono}
		d(q_n,\bar q)\longrightarrow 0 \text{ if and only if }\distcono(q_n,\bar q)\longrightarrow 0.
	\end{equation}

	We start by assuming that $d(q_n,\bar q)\longrightarrow 0$. 
	Then, there exists $\gamma_n:[0,1]\to M$ such that $\gamma_n(0)=q_n$ and $\gamma_n(1)=\bar q$ and $\int_0^1 \sqrt{\metr(\gamma_n(t),\gamma_n(t))}\,dt\longrightarrow 0$.
	This implies that
	\[
		|x_n|\le \int_0^1 \sqrt{\metr(\gamma_n(t),\gamma_n(t))}\,dt \longrightarrow 0,
	\]
	and thus that $x_n\longrightarrow 0$.
	This suffices to prove that $\distcono(q_n,\bar q)\longrightarrow 0$.

	On the other hand, if $\distcono(q_n,\bar q)\longrightarrow 0$, it suffices to consider the curves
	\[
		\gamma_n(t) = 
		\begin{cases}
			\big((1-2t)x_n, \theta_n\big)&\qquad\text{ if } 0\le t<\frac 1 2, \\
			\big(0,\theta_n+(2t-1)(\bar\theta-\theta_n)\big)&\qquad\text{ if } \frac 1 2\le t \le 1.
		\end{cases}
	\]
	Clearly $\gamma_n$ is Lipschitz and $\gamma_n(0)=q_n$ and $\gamma_n(1)=\bar q$.
	Finally, since $\metr|_{\sing}=0$, the proof is completed by
	\[
		d(q_n,\bar q)\le \int_0^1 \sqrt{\metr_{\gamma_n(t)}(\dot\gamma_n(t),\dot\gamma_n(t))}\,dt = \int_0^{\frac 1 2}\sqrt{\metr_{\gamma_n(t)}((-2x_n,0),(-2x_n,0))}\,dt = x_n \longrightarrow 0.
	\]
\end{proof}

\subsection{Surfaces of revolution}
\label{sec:surfaces}

Given two manifolds $M$ and $N$, endowed with two (possibly semi-definite) metrics $\metr^M$ and $\metr^N$, we say that $M$ is \em $C^1$-isometric \em to $N$ if and only if there exists a $C^1$-diffeomorphism $\Phi:M\to N$ such that $\Phi^*\metr_N=\metr_M$. 
Here $\Phi^*$ is the pullback of $\Phi$.
Recall that, in matrix notation, for any $q\in M$ it holds
\begin{equation}\label{eq:push}
	(\Phi^*\metr^N)_q(\xi,\eta)= \left(J_{\Phi}\right)^T \,\metr^M_{\Phi(q)} J_{\Phi} (\xi,\eta).
\end{equation}
Here $J_{\Phi}$ is the Jacobian matrix of $\Phi$.

We have the following.
\begin{prop}
	\label{prop:surfacerevolution}
	If $\alpha< -1$ the manifold $M_\alpha$ is $C^1$-isometric to a surface of revolution $\mathcal S=\{ (t, r(t)\cos\vartheta, r(t)\sin\vartheta ) \mid t\in\R,\, \vartheta\in\Tu\}\subset\R^3$ with profile $r(t)=|t|^{-\alpha}+\grando{t^{-2(\alpha+1)}}$ as $|t|\downarrow0$ (see figure~\ref{fig:revolutionsurface}), endowed with the metric induced by the embedding in $\R^3$.

	If $\alpha=-1$, $M_\alpha$ is globally $C^1$-isometric to the surface of revolution with profile $r(t)=t$, endowed with the metric induced by the embedding in $\R^3$.
\end{prop}

\begin{figure}
	\hfill \includegraphics[width=5cm]{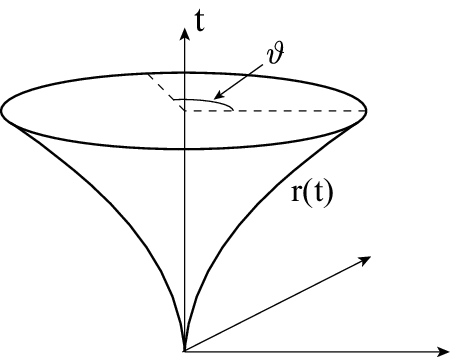}\hspace*{\fill}
	\caption{The surface of revolution of Proposition~\ref{prop:surfacerevolution} with $\alpha=-2$, i.e. $r(t)=t^2$.}
	\label{fig:revolutionsurface}
\end{figure}

\begin{proof}
For any $r\in C^1(\R)$, consider the surface of revolution $\mathcal S=\{ (t, r(t)\cos\vartheta, r(t)\sin\vartheta ) \mid t>0,\, \vartheta\in\Tu\}\subset\R^3$.
By standard formulae of calculus, we can calculate the corresponding (continuous) semi-definite Riemannian metric on $\mathcal S$ in coordinates $(t,\vartheta)\in\R\times\Tu$ to be
\begin{eqnarray*}
 {g}_{\mathcal S}(t,\vartheta) =\left(\ba{cc} 1+r'(t)^2&0\\0&r^2(t)\ea\right).
\end{eqnarray*}

Let now $\alpha<-1$ and consider the $C^1$ diffeomorphism $\Phi:(x,\theta)\in\R\times\Tu \mapsto (t(x),\vartheta(\theta))\in \mathcal S$ defined as the inverse of 
\begin{equation}\label{eq:ccoor}
	\Phi^{-1}(t,\vartheta)=
	\left(
	\begin{array}{c}
		x(t)\\
		\theta(\vartheta)
	\end{array}
	\right)
	=
	\left(
	\begin{array}{c}
		\int_0^t \sqrt{1+r'(s)^2}ds\\
		\vartheta
	\end{array}
	\right).
\end{equation}
Observe that $\Phi$ is well defined due to the fact that $r'$ is bounded near $0$.
Since $\partial_t (\Phi^{-1})=\partial_t x(t)=\sqrt{1+r'(t)^2}$, by \eqref{eq:push} the metric is transformed in
\begin{equation*}\label{eq:tildeg}
 	\Phi^* {g}_{\mathcal S} (x,\theta)= \left(J_{\Phi}^{-1}\right)^T {g}_{\mathcal S}(\Phi(x,\theta)) J_{\Phi}^{-1} = \left(\ba{cc} 1&0\\0&r\big(\Phi(x,\theta)\big)^{2}\ea\right).
\end{equation*}

We now claim that there exists a function $r\in C^{1}(\R)$ such that $r(t(x))=|x|^{-\alpha}$ near $\{x=0\}$.
Moreover, this function has expression
\begin{equation}
  \label{eq:rtexpr}
  	r(t)=
	\begin{cases}
	    t^{-\alpha}+\grando{t^{-2(\alpha+1)}}, &\quad \text{ if } t\ge0,\\
	    -(-t)^{-\alpha}+\grando{t^{-2(\alpha+1)}} &\quad \text{ if } t<0.
	\end{cases}
\end{equation}
Notice that, this function generates the same surface of revolution as $r(t)=|t|^{-\alpha}+\grando{t^{-2(\alpha+1)}}$, but is of class $C^{1}$ in $0$ while the latter is not.

The fact that $r(t(x))=|x|^{-\alpha}$ is equivalent to $r(t)=|x(t)|^{-\alpha}$, i.e.,
\begin{equation}
  \label{eq:integraleq}
  r(t) = \left( \int_0^t \sqrt{1+r'(s)^2}\,ds \right)^{-\alpha}.
\end{equation}
This integral equation has a unique solution.
Indeed, after algebraic manipulation and a differentiation, it is equivalent to the Cauchy problem
\begin{equation}
  \label{eq:peanoode}
  \begin{cases}
    r'(t) = \displaystyle\sqrt{\frac{1}{\alpha^{-2}r(t)^{-2(1+1/\alpha)} - 1}},\\
    r(0)=0.
  \end{cases}
\end{equation}
It is easy to check that, thanks to the assumption $\alpha<-1$, the r.h.s.\ of the ODE is Hölder continuous of parameter $1+1/\alpha$ at $0$ (but not Lipschitz).
This guarantees the existence of a solution, but not its unicity. 
Indeed, this equation admits two kinds of solutions, either $r_1(t)\equiv 0$ or $r_2(t)\not\equiv 0$, where the transition between $r_1(t)$ and $r_2(t-t_0)$ can happen at any $t_0\ge0$.
However, the only admissible solution of \eqref{eq:integraleq} is $r_2$, as can be directly checked.

We now prove the representation \eqref{eq:rtexpr}.
Assume w.l.o.g. that $t$, and hence $x(t)$, be positive.
Due to the Hölder continuity of the r.h.s.\ of the ODE in \eqref{eq:peanoode}, we get that $|r'(t)|\le C t^{1+1/\alpha}$.
Hence,
\[
|x'(t)-x'(0)| = |\sqrt{1+r'(t)^2}-1| \le |r'(t)| \le Ct^{1+1/\alpha}.
\]
Here, we used the $1/2$-H\"older property of the square root.
Finally, a simple computation shows that $|x(t)-t x'(0)| = \grando{t^{2+1/\alpha}}$, which yields
\begin{equation*}
	\label{eq:r}
    r(t)= \left( x(t)\right)^{-\alpha} = \left( t  +\grando{t^{2+1/\alpha}}\right)^{-\alpha}= t^{-\alpha} +\grando{t^{-(2+1/\alpha)(\alpha+1)}}=t^{-\alpha} +\grando{t^{-2(\alpha+1)}}.
\end{equation*}
Here, in the last step we used the fact that $-(2+1/\alpha)(\alpha+1)<-2(\alpha+1)$.
This proves the claim and thus the first part of the statement.

Let now $\alpha =-1$. 
In this case, by letting $r(t)=t$, the metric on the surface of revolution is
\begin{equation*}
 {g}_{\mathcal S}(t,\vartheta) =\left(\ba{cc} 2&0\\0&t^2\ea\right).
\end{equation*}
Consider the diffeomorphism $\Psi:(x,\theta)\in\R\times\Tu\mapsto (t,\vartheta)\in \mathcal S$ defined as 
\begin{equation}
	\label{eq:cintyiso}
	\Psi(x,\theta)=\sqrt 2\left(
	\begin{array}{c}
		x\\
		\theta
	\end{array}
	\right).
\end{equation}
Then the statement follows from the following computation,
\begin{equation*}
 	\Phi^* {g}_{\mathcal S} (x,\theta)= \left(J_{\Psi}^{-1}\right)^T {g}_{\mathcal S}(\Psi(x,\theta)) J_{\Psi}^{-1} = \left(\ba{cc} 1&0\\0& r\big(\Psi(x,\theta)\big)^{2}/2\ea\right)=\left(\ba{cc} 1&0\\0& x^2\ea\right).
\end{equation*}
\end{proof}

\begin{rmk}
	If $\alpha>-1$ we cannot have a result like the above, since the change of variables \eqref{eq:ccoor} is no more regular.
    In fact, the function $r(t)=t^{-\alpha}$ has an unbounded first derivative near $0$.
\end{rmk}

\section{Complex self-adjoint extensions} 
\label{app:schroedinger}

The natural functional setting for the Schr\"odinger equation on $M_\alpha$ is the space of square integrable complex-valued function $L^2_\C(M,\misura)$. 
Recall that a self-adjoint extension $B$ of an operator $A$ over $L^2_\C(M,\misura)$ is a \emph{real self-adjoint extensions} if and only if $u\in \dom B$ implies $\overline u \in \dom B$ and $B(\overline u)=\overline(Bu)$.
The self-adjoint extension of $A$ over $\ldue$ are exactly the restrictions to this space of the real self-adjoint extension of $A$ over $L^2_\C(M,\misura)$.

All the theory of Section~\ref{sec:sae} extends to the complex case, in particular, we have the following generalization of Theorem~\ref{thm:zettl}.

\begin{thm}[Theorem 13.3.1 in \cite{zettl}]\label{thm:zettlcompl}
	Let $A$ be the Sturm-Liouville operator on $L^2_\C(J,w(x)dx)$ defined in \eqref{eq:slo}.
	Then
	\[
		n_+(A)=n_-(A)=\#\{ \text{limit-circle endpoints of } J\}.
	\]

	Assume now that $n_+(A)=n_-(A)=2$, and let $a$ and $b$ be the two limit-circle endpoints of $J$. 
	Moreover, let $\phi_1,\phi_2\in \dmax(A)$ be linearly independent modulo $\dmin(A)$ and normalized by 
	$
		[\phi_1,\phi_2](a)=[\phi_1,\phi_2](b)=1.
	$
	Then, $B$ is a self-adjoint extension of $A$ over $L^2_\C(J,w(x)dx)$ if and only if $Bu=A^*u$, for any $u\in\dom B$, and one of the following holds
	\begin{enumerate}
		\item\label{it:disjointcompl} \emph{Disjoint dynamics:} there exists $c_+,c_-\in(-\infty,+\infty]$ such that $u\in\dom B$ if and only if
			\begin{gather*}
				[u,\phi_1](0^+)=c_{+}[u,\phi_2](0^+)\quad \text{and} \quad
				[u,\phi_1](0^-)=d_{+}[u,\phi_2](0^-).
			\end{gather*}
		\item\label{it:mixedcompl} \emph{Mixed dynamics:} 
			there exist $K\in SL_2(\R)$ and $\gamma\in(-\pi,\pi]$ such that $u\in\dom B$ if and only if
			\[
				U(0^-)= e^{i\gamma}K \,U(0^+), \qquad\text{ for } 
					U(x)= \left( \begin{array}{c} {[u,\phi_1](x)}\\ {[u,\phi_2](x)}\\ \end{array} \right).
			\]
	\end{enumerate}
	Finally, $B$ is a real self-adjoint extension if and only if it satisfies \eqref{it:disjointcompl} the disjoint dynamic or \eqref{it:mixedcompl} the mixed dynamic with $\gamma=0$.
\end{thm}

As a consequence of Theorem~\ref{thm:zettlcompl}, we get a complete description of the essential self-adjointness of $\delcomp$ over $L^2_\C(M,\misura)$, extending Theorem~\ref{thm:sa}, and of the complex self-adjoint extensions of $\delfz$, extending Theorem~\ref{thm:sae}.

\begin{thm}\label{thm:sacompl}
    Consider $M_\alpha$ for $\alpha \in\R$ and the corresponding Laplace-Beltrami operator $\delcomp$ as an unbounded operator on $L^2_\C(M,\misura)$.
    Then it holds the following.
    \begin{enumerate}
    \renewcommand{\theenumi}{\roman{enumi}}
        \item If $\alpha\le -3$ then $\delcomp$ is essentially self-adjoint;
        \item if $\alpha\in(-3,-1]$, only the first Fourier component $\widehat \Delta_0$ is not essentially self-adjoint;
        \item if $\alpha\in(-1,1)$, all the Fourier components of $\delcomp$ are not essentially self-adjoint;
        \item if $\alpha\ge1$ then $\delcomp$ is essentially self-adjoint.
    \end{enumerate}
\end{thm}

\begin{thm} \label{thm:saecompl}
	Let $\dminf$ and $\dmaxf$ be the minimal and maximal domains of $\delfzcomp$ on $L^2_\C(\R\setminus\{0\},|x|^{-\alpha})$, for $\alpha\in (-3,1)$.
	Then, 
	\begin{gather*}
		\dminf =  \text{closure of } \cc {\R\setminus\{0\}} \text{ in } H^2_\C(\R\setminus\{0\},|x|^{-\alpha}dx)\\
		\dmaxf = \{ u=u_0+u^+_D\phidp + u^+_N \phinp + u^-_D\phidm + u^-_N \phinm \colon\: u_0\in\dminf \text{ and } u^\pm_D,\,u_N^\pm\in\C\},
	\end{gather*}
	Moreover, $A$ is a self-adjoint extension of $\delfz$ if and only if $Au=(\delfz)^*u$, for any $u\in\dom A$, and one of the following holds
	\begin{enumerate}
 	\renewcommand{\theenumi}{\roman{enumi}}
		\item \em Disjoint dynamics: \em there exist $c_+,c_-\in (-\infty,+\infty]$ such that 
			\[
				\dom A = \big\{ u\in\dmaxf\colon\: u_N^+=c_+u^+_D \text{ and } u_N^-=c_- u^+_D\big\}.
			\]
		\item \em Mixed dynamics: \em there exist $K\in SL_2(\R)$ and $\gamma\in(-\pi,\pi]$ such that
			\[
				\dom A = \big\{ u\in\dmaxf \colon\: (u_D^-,u_N^-)= e^{i\gamma} K \, (u_D^+,u_N^+)^T \big\}.
			\]
	\end{enumerate}
	Finally, the Friedrichs extension $(\delfz)_F$ is the one corresponding to the disjoint dynamics with $c_+=c_-=0$ if $\alpha\le -1$ and with $c_+=c_-=+\infty$ if $\alpha>-1$.
\end{thm}

\section*{Acknowledgments}
The authors would like to thank Professors G. Dell'Antonio, A.Grigor'yan, G. Panati and A. Posilicano, as well as M. Morancey, for the helpful discussions.

\bibliographystyle{amsplain}
\bibliography{coni}

\end{document}